\documentclass[a4paper,10pt]{amsart}

\usepackage[utf8]{inputenc}
\usepackage{mathtools}

\usepackage{amsmath}
\usepackage{amstext}
\usepackage{amsfonts}
\usepackage{amsthm}
\usepackage{amssymb}

\usepackage{hyperref}
\hypersetup{colorlinks}

\newcommand{\R}{{\mathbb{R}}}

\newcommand{\E}{{\mathbb{E}}}
\newcommand{\N}{{\mathbb{N}}}

\newcommand{\F}{{\mathcal{F}}}  
\renewcommand{\P}{{\mathbb{P}}} 

\newcommand{\B}{{\mathcal{B}}}
\newcommand{\LB}{{\mathcal{L}}}
\newcommand{\Tr}{{\textrm{Tr}}}
\newcommand{\diff}[1]{\,\mathrm{d}#1}

\newcommand{\ee}{\mathrm{e}}
\newcommand{\id}{\mathrm{id}}
\newcommand{\p}{q} 				

\theoremstyle{plain}

\newtheorem{definition}{Definition}[section]
\newtheorem{theorem}[definition]{Theorem}
\newtheorem{lemma}[definition]{Lemma}

\newtheorem{prop}[definition]{Proposition}
\newtheorem{assumption}[definition]{Assumption}

\theoremstyle{definition}
\newtheorem{remark}[definition]{Remark}

\newtheorem{example}[definition]{Example}

\begin{document}

\title[BDF2-Maruyama for SPDEs]
{The BDF2-Maruyama Scheme for Stochastic Evolution Equations with Monotone Drift}

\author[R.~Kruse]{Raphael Kruse}
\address{Raphael Kruse\\
Martin-Luther-Universit\"at Halle-Wittenberg\\
Institut f\"ur Mathematik\\
06099 Halle (Saale)\\
Germany}
\email{raphael.kruse@mathematik.uni-halle.de}

\author[R.~Weiske]{Rico Weiske}
\address{Rico Weiske\\
Martin-Luther-Universit\"at Halle-Wittenberg\\
Institut f\"ur Mathematik\\
06099 Halle (Saale)\\
Germany}
\email{rico.weiske@mathematik.uni-halle.de}


\begin{abstract}
  We study the numerical approximation of stochastic evolution equations with a
  monotone drift driven by an infinite-dimensional Wiener process. To
  discretize the equation, we combine a drift-implicit two-step BDF method for
  the temporal discretization with an abstract Galerkin method for the spatial
  discretization. After proving well-posedness of the BDF2-Maruyama scheme, we
  establish a convergence rate of the strong error for equations under suitable
  Lipschitz conditions. We illustrate our theoretical results through various
  numerical experiments and compare the performance of the BDF2-Maruyama scheme
  to the backward Euler--Maruyama scheme.
\end{abstract}

\maketitle

\section{Introduction}
\label{sec:intro}

\noindent
In this paper, we investigate a spatio-temporal discretization of
a class of nonlinear stochastic evolution equations with monotone drift.
To be more precise, let $(\Omega,\F,(\F_t)_{t \in [0,T]},\P)$ be a filtered
probability space satisfying the usual conditions for fixed $T \in (0,\infty)$.
By $W$ we denote an infinite-dimensional Wiener process with covariance
operator $Q$ which is $(\F_t)_{t \in [0,T]}$-adapted and takes values in a
separable Hilbert space $U$. The stochastic evolution equation under
consideration then reads 
\begin{equation}
  \label{sde_problem}
  \diff X(t) +  A(X(t)) \diff t = B(X(t)) \diff W(t)
  \quad \text{on } [0,T], \quad X(0) = X_0,
\end{equation}
where the operators $A \colon V \times \Omega \rightarrow V^*$ and
$B \colon V \times \Omega \rightarrow \LB_2(Q^\frac12(U),H)$ are defined on a
Gelfand triple $V \hookrightarrow H \cong H^* \hookrightarrow V^*$ for a real,
reflexive, separable Banach space $V$ and a real, separable Hilbert space $H$.
The initial value satisfies $X_0 \in L^2(\Omega,\F_0,\P;H)$, while
the stochastic integral in \eqref{sde_problem} is understood in the sense of
the stochastic It\=o-calculus. This setting allows us to treat several linear,
semi-linear, and quasi-linear stochastic partial differential equations in a
unified analytical framework, see \cite[Chapter~1]{prevot2007} for more
explicit examples.

Throughout this paper we employ the \emph{variational approach} from
\cite{krylov1981, liu2015, prevot2007, rozovskii1990} to analyze the solution
to the stochastic evolution equation \eqref{sde_problem}
and its numerical approximation. We essentially impose the 
same assumptions on the operators $A$ and $B$ as in
\cite{krylov1981,prevot2007}, which are sufficient to ensure the existence of a
unique strong solution to \eqref{sde_problem}. In particular, we assume that
the operators $A$ and $B$ satisfy a monotonicity condition 
and a coercivity condition (see \eqref{con:mono AB} and \eqref{con:coer AB}).
We refer to Section~\ref{sec:discret} for a full account of all imposed
conditions and the precise definition of the exact solution to
\eqref{sde_problem}.

The numerical approximation of \eqref{sde_problem} with time-dependent
operators $A$ and $B$ was studied under similar assumptions in
\cite{gyoengy2005,gyoengy2009}. It was proven in \cite{gyoengy2005} that the
spatio-temporal approximations arising from the forward and backward
Euler--Maruyama method combined with an abstract Galerkin method converge
weakly to the exact solution. Furthermore, convergence rates for the strong
error of these methods were derived in \cite{gyoengy2009} under additional
regularity assumptions on the exact solution and the operator $A$.
Notice that the spatial semi-discretization of \eqref{sde_problem} can lead
to a high-dimensional stiff system of stochastic ordinary differential
equations. In practical simulations it is therefore beneficial
to use an \emph{A-stable} numerical method
for the temporal discretization, such as the backward Euler--Maruyama method.
These methods typically avoid severe step size restrictions as, for instance,
\cite[condition (2.16)]{gyoengy2005} for the forward Euler--Maruyama method.
We refer to \cite{buckwarsickenberger2012, higham2000} for a general discussion
of A-stable numerical schemes for stiff stochastic differential equations.

In this paper we focus on the BDF2 method with an equidistant step size for the
temporal discretization.
The BDF2 method and the backward Euler method both belong to the family of
backward differentiation formulas (BDF) which have proved effective for the
approximation of stiff ordinary differential equations, see, e.g.,
\cite{hairer1996, strehmel2012}. In particular, if applied to ODEs the BDF2
method has the same computational cost and enjoys the same stability properties
as the backward Euler method, while having the advantage of a higher order of
convergence. 

The discretization of stochastic ordinary differential equations (SODE)
by means of the BDF2 method has already been studied in the literature.
The mean-square convergence of drift-implicit linear two-step Maruyama methods
on equidistant time grids was investigated in \cite{buckwar2006} under
a global Lipschitz condition on the coefficients.
Moreover, higher convergence rates of such methods were
derived for problems driven by small noise. The mean-square stability and
convergence for general drift-implicit linear multi-step methods on
non-equidistant time grids were further examined in \cite{sickenberger2008}.
In addition, the mean-square convergence of the BDF2-Maruyama scheme
was proven under a monotonicity condition on the coefficients in
\cite{andersson2017}. However, to the best of our knowledge, multi-step methods
for the temporal discretization of nonlinear stochastic evolution equations
have not been investigated in detail yet.

To formulate the numerical approximation of \eqref{sde_problem}, we consider an
equidistant temporal grid with step size $k = \frac{T}{N_k}$, $N_k \in \N$,
and grid points $t_n = nk$ for $n \in \{0,\ldots,N_k\}$. In addition, let
$V_h \subset V$ be a finite dimensional subspace depending on some parameter
$h \in (0,1)$. For given initial values $(X_{k,h}^n)_{n=0,1}$ the
\emph{BDF2-Maruyama scheme} is defined by 
\begin{equation}
\label{BDF2 scheme}
\begin{aligned}
&\Big( \frac{3}{2}X^n_{k,h}-2X^{n-1}_{k,h}+\frac{1}{2}X^{n-2}_{k,h},
v \Big)_H + k \langle A(X^n_{k,h}), v \rangle_{V^* \times V}\\
&\qquad = \Big( \frac{3}{2}B(X^{n-1}_{k,h})\Delta_k W^n
-\frac{1}{2}B(X^{n-2}_{k,h})\Delta_k W^{n-1},v \Big)_H
\quad \P\text{-a.s.}
\end{aligned}
\end{equation}
for all $v \in V_h$ and $n=2,\ldots,N_k$, where we define the Wiener increments
by $\Delta_k W^n \coloneqq W(t_n)-W(t_{n-1})$ for $n \in \{1,\ldots,N_k\}$.
In order to generate suitable initial values for the scheme \eqref{BDF2 scheme},
the \emph{backward Euler--Maruyama method} will come in handy and is defined
for given initial value $X_{k,h}^0$ by
\begin{equation}
\label{BEM scheme}
\big( X^n_{k,h}-X^{n-1}_{k,h}, v \big)_H
+ k \langle A(X^n_{k,h}), v \rangle_{V^* \times V}
= \big(B(X^{n-1}_{k,h})\Delta_k W^n,v \big)_H
\quad \P\text{-a.s.}
\end{equation}
for all $v \in V_h$ and $n=1,\ldots,N_k$.

As our first main result we show that the discrete process
$(X_{k,h}^n)_{n=0}^{N_k}$ is indeed well-defined by \eqref{BDF2 scheme} under
essentially the assumptions used in \cite{krylov1981, prevot2007}. Further,
under additional conditions on $A$ and $B$ and the temporal regularity of the
exact solution to \eqref{sde_problem}, cf. Assumption~\ref{assump AB2}
to Assumption~\ref{assump X}, we also show that $(X_{k,h}^n)_{n=0}^{N_k}$
is convergent to the exact solution $X$ in the following sense: There exist
$C \geq 0$, $\p \in [2,\infty)$ and $\gamma \in (0,\infty)$, where $\p$
is determined by the regularity of $X$ and $\gamma$ arises from an
approximation error related to the initial values, such that for every
sufficiently small temporal step size $k$ and every $h \in (0,1)$ it holds
\begin{equation}
  \label{ineq1:err est}
  \begin{aligned}
    &\max_{n\in\{2,\ldots,N_k\}} \|X^n_{k,h}-X(t_n)\|_{L^2(\Omega;H)}
    + \Big(k \sum_{n=2}^{N_k}
    \|X^n_{k,h}-X(t_n)\|_{L^2(\Omega;V)}^2 \Big)^{\frac{1}{2}}\\
    &\quad \leq C \bigg[ k^\frac{1}{\p}
    + h^\frac{\gamma}{2} + \max_{n\in\{2,\ldots,N_k\}}
    \|(P_h-\id)X(t_n)\|_{L^2(\Omega;H)}\\
    &\qquad \quad + \big(1+\|P_h\|_{\LB(V)}\big) \Big(k \sum_{n=2}^{N_k}
    \|(R_h-\id)X(t_n)\|_{L^2(\Omega;V)}^2 \Big)^{\frac{1}{2}} \bigg].
  \end{aligned}
\end{equation}
Hereby, $P_h \colon H \to V_h$ denotes the orthogonal projection on $V_h$ with
respect to the inner product in $H$ and $R_h \colon V \to V_h$ is mapping to
the best approximation in $V_h$ with regard to the norm in $V$. This error
estimate is precisely stated in Theorem~\ref{thm:err}. Notice that the order
of convergence also depends on the chosen Galerkin method.

Let us emphasize some important features of our error analysis: First, we do
not apply It\=o's formula since we want to avoid the difficult task to
interpolate the approximation of the two-step BDF2-Maruyama scheme to
continuous time. Second, in contrast to
\cite{gyoengy2009}, we also do not require a priori knowledge of higher spatial
regularity of the exact solution beyond the Gelfand triple $(V,H,V^*)$ since
such regularity results are often not available in the literature and difficult
to verify for nonlinear stochastic evolution equations. Finally, as already
mentioned above, we cannot avoid imposing additional assumptions on the temporal
regularity of the exact solution. However, we only require that the exact
solution has a finite $q$-variation norm (see \eqref{def:p_var})
instead of the (slightly) stronger H\"older continuity condition typically used
in the literature.

As it was observed in \cite{emmrich2009} for deterministic evolution equations,
the following identity plays an important role in the error and stability
analysis of the BDF2 scheme. For all $x_1,x_2,x_3 \in H$ it holds true
that
\begin{equation}
  \label{eqn:id 2-step}
  \begin{aligned}
    4 \Big( \frac{3}{2}x_3-2x_2+\frac{1}{2}x_1, x_3 \Big)_H
    &= \|x_3\|^2_H - \|x_2\|^2_H + \|2x_3-x_2\|^2_H\\
    &\qquad - \|2x_2-x_1\|^2_H + \|x_3-2x_2+x_1\|^2_H.
  \end{aligned}
\end{equation}
This identity also has been utilized in \cite{andersson2017} to derive a strong
convergence rate of the BDF2-Maruyama scheme applied to SODEs. It will also
be crucial to prove \eqref{ineq1:err est}.

The paper is structured as follows. In Section~\ref{sec:prelimi}, we introduce
some notation and recall important concepts related to the abstract analytical
framework, the stochastic integration and the approximation in
infinite-dimensional spaces. Section~\ref{sec:discret} is devoted to
establishing sufficient conditions for the existence of a unique solution to
\eqref{sde_problem} and showing the well-posedness of the BDF2-Maruyama scheme
\eqref{BDF2 scheme} under these conditions. Moreover, we present the stochastic
heat equation as an applicable example. In Section~\ref{sec:converg} we prove
the error estimate \eqref{ineq1:err est} under additional regularity
assumptions. Finally, in Section~\ref{sec:simu}, we provide two numerical
experiments to illustrate our theoretical results and discuss aspects of their
implementation. In particular, the comparison of the schemes \eqref{BDF2 scheme}
and \eqref{BEM scheme} in the temporal error analysis indicates that the
BDF2-Maruyama scheme is favourable for problems driven by noise with higher
spatial regularity or noise with small intensity.

\section{Preliminaries}
\label{sec:prelimi}

In this section, we briefly recall some basic concepts from functional
analysis, stochastic analysis, and numerical analysis which are used throughout
this paper. Mostly, we employ the same notation as in
\cite[Chapter~2]{liu2015} and \cite[Chapter~2]{prevot2007}.

Let $(H,(\cdot,\cdot)_H,\|\cdot\|_H)$ be a real, separable Hilbert space and
let $(V,\|\cdot\|_V)$ be a real, reflexive and separable Banach 
space that is continuously and densely embedded in $H$. We denote the dual
spaces of $H$ and $V$ by $H^*$ and $V^*$, respectively, and use 
$\langle \cdot,\cdot \rangle_{V^* \times V}$ for the dual pairing between $V$
and its dual $V^*$. We consider the Gelfand triple $(V,H,V^*)$ which satisfies
$V \hookrightarrow H \cong H^* \hookrightarrow V^*$ with $\hookrightarrow$
denoting dense and continuous embeddings and $\cong$ the identification of $H$
with its dual space in terms of the Riesz isomorphism. In particular, there
exists $\beta_{V \hookrightarrow H} \in (0,\infty)$ such that
for every $v \in V$ the inequality $\|v\|_H \leq \beta_{V \hookrightarrow H}
\|v\|_V$ holds. In addition, we recall that
\begin{align*}
  \langle u, v \rangle_{V^\ast \times V} = (u,v)_H 
\end{align*}
holds for all $u \in H$ and all $v \in V$.

For $T \in (0,\infty)$ let
$(\Omega,\F,(\F_t)_{t \in [0,T]},\P)$ be a filtered 
probability space satisfying the usual conditions. For $p \in [1,\infty)$
we denote by $L^p(\Omega;V) \coloneqq L^p(\Omega,\mathcal{F},\P;V)$ and
$L^p([0,T] \times \Omega;V) \coloneqq L^p([0,T] \times \Omega,
\mathcal{B}([0,T]) \otimes \mathcal{F}, \diff{t} \otimes \P;V)$
the Bochner--Lebesgue spaces which are, respectively, endowed with the norms
\begin{align*}
  \|X\|_{L^p(\Omega;V)} \coloneqq \big( \mathbb{E} \big[\|X\|_V^p\big]
  \big)^{\frac{1}{p}} \quad \text{and} \quad
  \|X\|_{L^p([0,T] \times \Omega;V)} \coloneqq \Big(
  \mathbb{E} \Big[\int_0^T \|X(t)\|_V^p \diff t \Big]\Big)^{\frac{1}{p}}.
\end{align*}
For an introduction to Bochner--Lebesgue spaces we refer, e.g., to
\cite[Appendix~E]{cohn2013} and \cite[Section~4.2]{papageorgiou2018}.

Next, let $(U,(\cdot,\cdot)_U)$ be a further separable Hilbert space and denote
by $\LB(U,H)$ the Banach space of all linear, bounded
operators from $U$ to $H$. By $\LB_2(U,H)$ we then denote 
the Hilbert space of all 
operators $B \in \LB(U,H)$ with finite Hilbert--Schmidt norm
$\|B\|^2_{\LB_2(U,H)} \coloneqq \Tr(B^\ast B)$. Moreover, for every
non-negative, symmetric operator $Q \in \LB(U) = \LB(U,U)$ there exists a
unique operator $Q^{\frac{1}{2}} \in \LB(U)$ satisfying $Q = Q^{\frac{1}{2}}
\circ Q^{\frac{1}{2}}$.
Then, $U_0 \coloneqq Q^{\frac{1}{2}}(U)$ defines a
Hilbert space if endowed with the inner product
\begin{equation}
  \label{eq2:U0}
  (u,v)_{U_0} \coloneqq (Q^{-\frac{1}{2}}u,Q^{-\frac{1}{2}}v)_U
  \quad \forall u,v \in U_0,
\end{equation}
where $Q^{-\frac{1}{2}}$ denotes the pseudo-inverse of $Q^{\frac{1}{2}}$. For
further details, we refer to \cite[Section~2.3]{liu2015} and
\cite[Section~2.3]{prevot2007}. 

For a given symmetric and non-negative operator $Q \in \LB(U)$ we then 
denote by $W$ a Hilbert space valued Wiener process with respect to the
filtration $(\F_t)_{t \in [0,T]}$ with covariance operator 
$Q$ as defined in \cite[Section~2.1]{prevot2007}. If the covariance operator $Q
\in \LB(U)$ is, in addition, of finite trace, 
then we recall that the Wiener process $W$ takes values in $U$ almost surely
and it has the representation 
\begin{equation}
  \label{eq2:KHexp}
  W(t) = \sum_{j=1}^\infty \sqrt{q_j} \chi_j \beta_j(t), \quad t \in [0,T],
  \quad \P\text{-a.s.}
\end{equation}
Hereby, $\{\chi_j\}_{j \in \N}$ is an orthonormal basis of $U$ which consists of
eigenvectors of $Q$ with summable eigenvalues $q_j \geq 0$ and
$\{\beta_j\}_{j \in \N}$ is a family of independent scalar Brownian motions.

Further, for a given stochastically integrable process $\Phi \colon [0,T]
\times \Omega \to \LB_2(U_0,H)$ we denote the stochastic It\=o-integral of
$\Phi$ by 
\begin{align*}
  \int_0^T \Phi(t) \diff{W(t)}.
\end{align*}
For the construction of a Hilbert space valued Wiener process, the
stochastic It\=o-integral and their properties we again refer to
\cite[Chapter~2]{liu2015}, \cite[Chapter~2]{prevot2007}, as well as
\cite[Chapter~4]{daprato1992}. 
Moreover, we recall from \cite[Section~2.5]{prevot2007} and 
\cite[Section~4.3]{daprato1992} that the construction
of the stochastic It\=o-integral can be extended to the case of a cylindrical
Wiener process whose covariance operator $Q$ is not necessarily of finite
trace.

To measure the regularity of the trajectories of a continuous $V$-valued
stochastic process, we use the concept of finite $\p$-variation for a given
$\p \in [1,\infty)$. A continuous function $f \colon [0,T] \to V$ is of
finite $\p$-variation with respect to the norm $\|\cdot\|_V$ if 
\begin{equation}
  \label{def:p_var}
  \|f\|_{\p-\mathrm{var},V} \coloneqq
  \Big( \sup_\mathcal{P} \sum_{n=1}^{N_\mathcal{P}} \|f(t_n)-f(t_{n-1})\|^\p_V
  \Big)^{\frac{1}{\p}}
   < \infty,
\end{equation}
where the supremum is taken over the set of all finite partitions $\mathcal{P}
= \{t_0,\ldots,t_{N_\mathcal{P}}\}$ of the interval $[0,T]$. For further
details on the concept of the $\p$-variation we refer to
\cite[Section~5]{friz2010} and \cite[Section~1]{lyons2007}.

Finally, we recall some properties of abstract Galerkin methods. 
Let $(V_h)_{h \in (0,1)}$ be a family of finite dimensional subspaces of the
Banach space $V \subset H$ such that for every $v \in H$ it holds $\inf_{v_h
\in V_h} \|v_h-v\|_H \rightarrow 0$ as $h \rightarrow 0$. Such 
a family of subspaces is called an (abstract) \emph{Galerkin scheme}. By $N_h
\in \N$ we denote the dimension of the subspace $V_h$. In addition, the
parameter $h \in (0,1)$ governs the granularity of the Galerkin scheme. In
particular, if $H$ is of infinite dimensions then $N_h = \dim(V_h) \to \infty$ as
$h \to 0$. 

Further, we define $P_h \colon H \rightarrow V_h$ as the orthogonal projection
map onto $V_h$ with respect to the inner product $(\cdot,\cdot)_H$. Therefore,
for each $v \in H$ the element $P_h v \in V_h$ is the best approximation of $v$
in $V_h$ with respect to the norm in $H$, see, e.g.,
\cite[Theorem~5.2]{brezis2011}. Hence, it holds  
\begin{equation*}
  \|P_h v-v\|_H = \mathrm{dist}_H(v,V_h)
  \coloneqq \inf_{v_h \in V_h} \|v_h-v\|_H
  \quad \forall v \in H.
\end{equation*}
If the Banach space $V$ is uniformly convex then for each $v \in V$ there
also exists a unique element $R_h v \in V_h$ that is the best approximation
of $v$ in $V_h$ with respect to the norm in $V$, see
\cite[Exercise~3.32]{brezis2011}. This defines a (possibly nonlinear) map  
$R_h \colon V \rightarrow V_h$, $v \mapsto R_h v$ satisfying
\begin{equation*}
  \|R_h v-v\|_V = \mathrm{dist}_V(v,V_h) 
  \coloneqq \inf_{v_h \in V_h} \|v_h - v\|_V 
  \quad \forall v \in V.
\end{equation*}
Recall that if $V$ is itself a Hilbert space, then it is also uniformly convex,
see \cite[Section~3.7]{brezis2011}. In this case, the mapping $R_h$ coincides 
with the orthogonal projector of $V$ onto $V_h$ with respect to the 
inner product of $V$.

\section{Discretization: a priori estimates and well-posedness}
\label{sec:discret}

The goal of this section is to establish sufficient conditions on the operators
$A$ and $B$ and the initial conditions to ensure the well-posedness of the
numerical scheme \eqref{BDF2 scheme}. For this, we first establish an a priori
estimate for solutions to the numerical scheme for any
value of the spatial refinement parameter $h \in (0,1)$ and every sufficiently
small temporal step size $k = \frac{T}{N_k}$, $N_k \in \N$. Afterwards we also
discuss existence and uniqueness of a solution to this scheme.

Throughout this section, we fix $p \in (1,\infty)$ and a Gelfand triple
$(V,H,V^\ast)$ as in Section~\ref{sec:prelimi}. 

\begin{assumption}
  \label{assump AB1}
  The operators $A \colon V \times \Omega \rightarrow V^*$ and $B \colon V \times
  \Omega \rightarrow \LB_2(U_0,H)$ are measurable with respect to
  $\B(V)\otimes\F_0/\B(V^*)$ and $\B(V)\otimes\F_0/\B(\LB_2(U_0,H))$,
  respectively, where $\B(V)$ denotes the Borel $\sigma$-algebra on $V$. In
  addition, the operator $A$ 
  is hemicontinuous, i.e. the mapping $z \colon [0,1] \to \R$, $\lambda \mapsto
  \langle A(u+ \lambda v,\omega), w \rangle_{V^* \times V}$ is continuous for all
  $u,v,w \in V$ and $\omega \in \Omega$. Moreover, there are
  $\kappa, c \in [0,\infty)$, $\mu \in (0,\infty)$ and $\nu \in [1,\infty)$ such
  that the operators $A$ and $B$ satisfy the monotonicity condition
  \begin{equation}
  \label{con:mono AB}
    \begin{aligned}
    2 \langle A(u)-A(v),u-v \rangle_{V^* \times V} + \kappa \|u-v\|_H^2
    \geq \|B(u)-B(v)\|_{\LB_2(U_0,H)}^2 \quad \text{on } \Omega
  \end{aligned}
  \end{equation}
  for all $u, v \in V$ and the coercivity condition
  \begin{equation}
  \label{con:coer AB}
    2 \langle A(v),v \rangle_{V^* \times V} + \kappa \|v\|_H^2
    \geq \nu \|B(v)\|_{\LB_2(U_0,H)}^2 + \mu \|v\|_V^p - c \quad 
    \text{on } \Omega
  \end{equation}
  for all $v \in V$.
  Furthermore, the growth condition
  \begin{equation} \label{con:growth A}
    \| A(v) \|_{V^*} \leq c (1+\| v \|_V)^{p-1} \quad \text{on } \Omega
  \end{equation}
  is satisfied for all $v \in V$.
\end{assumption}

Before we turn to the numerical scheme \eqref{BDF2 scheme}, we
mention that Assumption~\ref{assump AB1} is sufficient to ensure the
existence of a uniquely determined exact solution to \eqref{sde_problem}, which
we define in the same way as in \cite[Definition~4.2.1]{prevot2007}. More
precisely, let $X_0 \in L^2(\Omega,\F_0,\P;H)$ be the initial value. 
Then, we call a continuous, $H$-valued and $(\F_t)_{t \in [0,T]}$-adapted
process $X \in L^p([0,T] \times \Omega;V) \cap L^2([0,T] \times \Omega;H)$
a solution of \eqref{sde_problem} if
\begin{equation}
\label{sde solution}
  X(t) + \int_0^t  A(\bar{X}(s)) \diff s
  = X_0 + \int_0^t B(\bar{X}(s)) \diff W(s)
\end{equation}
holds in $V^\ast$ for all $t \in [0,T]$ almost surely, where $\bar{X}$ is a
$V$-valued, progressively measurable modification of $X$. Such a solution is
said to be unique if any two solutions $X$ and $Y$ to \eqref{sde_problem} are
indistinguishable, i.e.,
\begin{equation}
  \label{eq3:indistinguishable}
  \P \Big( \sup_{t \in [0,T]} \|X(t)-Y(t)\|_H = 0 \Big) = 1.
\end{equation}
For a proof of the following result, we refer to \cite[Section~3]{krylov1981}
and \cite[Chapter~4]{prevot2007}.

\begin{prop}
  \label{prop:exactsol}
  Let Assumption~\ref{assump AB1} be satisfied for some $p \in (1,\infty)$ and
  let $X_0 \in L^2(\Omega,\F_0,\P;H)$. Then the stochastic evolution equation
  \eqref{sde_problem} admits a unique solution.
\end{prop}

We now turn to the question of well-posedness of the numerical scheme 
\eqref{BDF2 scheme}. As for every two-step scheme it is first necessary to 
find two suitable initial values. The following assumption
is required to ensure adaptedness and square-integrability of the numerical
solution.

\begin{assumption}
  \label{assump IC}
  The initial values $(X^n_{k,h})_{n=0,1}$ satisfy
  \begin{equation*}
    X^n_{k,h} \in L^2(\Omega,\F_{t_n},\P;H) \quad \text{ and } \quad
    B(X^n_{k,h}) \in L^2(\Omega, \F_{t_n}, \P; \LB_2(U_0,H)), \quad
    n \in \{0,1\},
  \end{equation*}
  and $\P(\{ \omega \in \Omega \, : \, X^n_{k,h}(\omega) \in V_h\} ) = 1$
  for each $n \in \{0,1\}$.
\end{assumption}

The BDF2-Maruyama scheme is well-defined if there exists a unique discrete
stochastic process $(X^n_{k,h})_{n=0}^{N_k}$, which is
$(\F_{t_n})_{n=0}^{N_k}$-adapted, $\P$-almost surely $V_h$-valued and solves
the recursion \eqref{BDF2 scheme}. We call such a solution unique if any two
solutions $(X_{k,h}^n)_{n=0}^{N_k}$ and $(Y_{k,h}^n)_{n=0}^{N_k}$ to
\eqref{BDF2 scheme} are indistinguishable, which is understood in the same way
as in \eqref{eq3:indistinguishable}. For the purpose of readability, we omit
the dependence of the discrete solution on the parameters $k$ and $h$
by writing $X^n \coloneqq X^n_{k,h}$ throughout the proofs presented
in Section~\ref{sec:discret} and Section~\ref{sec:converg}.

Before we prove the existence of a unique solution to \eqref{BDF2 scheme}, 
we first derive the following useful \emph{a priori estimate}.

\begin{theorem}
  \label{thm:a priori}
  Let Assumption~\ref{assump AB1} be satisfied for some $p \in (1,\infty)$.
  Let $h \in (0,1)$ and  $k = \frac{T}{N_k}$, $N_k \in \N$, be fixed with
  $2k\kappa < 1$. Let $(X_{k,h}^n)_{n=0}^{N_k}$ be an arbitrary
  $(\F_{t_n})_{n=0}^{N_k}$-adapted and $\P$-almost surely $V_h$-valued process
  satisfying Assumption~\ref{assump IC} and \eqref{BDF2 scheme}.
  Then, it holds
  \begin{equation}
  \label{ineq:a priori}
    \begin{aligned}
    &\max_{n\in\{2,\ldots,N_k\}}\Big( \E \big[\|X^n_{k,h}\|^2_H\big]
    + 2k\nu \E \big[\|B(X^n_{k,h})\|^2_{\LB_2(U_0,H)}\big] \Big)\\
    &\qquad + 2k\mu \sum_{n=2}^{N_k} \E \big[\|X^n_{k,h}\|^p_V\big]
    + \frac{\nu-1}{\nu} \sum_{n=2}^{N_k} \E \big[\|X^n_{k,h}-2 X^{n-1}_{k,h} +
    X^{n-2}_{k,h}\|^2_H\big]\\   
    &\quad \leq C_k \Big( T + \sum_{n=0}^{1} \E \big[\|X^n_{k,h}\|_H^2\big]
    + k \sum_{n=0}^{1} \E \big[\|B(X^n_{k,h})\|^2_{\LB_2(U_0,H)}\big] \Big),
    \end{aligned}
  \end{equation}
  where $C_k = 2\max\{2c,11,2\nu\} (1-2k\kappa)^{-1}
  \ee^{2\kappa T (1-2k\kappa)^{-1}}$.
\end{theorem}

\begin{proof}
  By an inductive argument we will show that for every $j \in \{1,\ldots,N_k\}$ 
  it holds
  \begin{equation}
  \label{ineq:ap}
    \begin{aligned}
    &\E \big[\|X^j \|^2_H\big]
    + 2k\nu \E \big[\|B(X^j)\|^2_{\LB_2(U_0,H)}\big]
    + 2k\mu \sum_{n=2}^j \E \big[\|X^n \|^p_V\big] \\   
    &\qquad + \frac{\nu-1}{\nu} \Big(
    \E \big[\|2X^j -X^{j-1}\|^2_H\big] + \sum_{n=2}^j
    \E \big[\|X^n -2 X^{n-1} + X^{n-2} \|^2_H\big] \Big)\\    
    &\quad \leq \frac{C_k}{2}
    \Big( T + \sum_{n=0}^{1} \E \big[\|X^n \|_H^2\big]
    + k \sum_{n=0}^{1} \E \big[\|B(X^n)\|^2_{\LB_2(U_0,H)}\big] \Big).
    \end{aligned}
  \end{equation} 
  where we set the two sums on the left-hand side equal to zero in the case
  $j = 1$. Observe that \eqref{ineq:ap} directly implies the estimate 
  \eqref{ineq:a priori}.  Moreover, it immediately follows from
  Assumption~\ref{assump IC} and the choice of $C_k$ that \eqref{ineq:ap} 
  holds true for $j=1$.  
  
  Next, let us assume that the estimate \eqref{ineq:ap} holds true 
  for some fixed $j-1 \in \{1,\ldots,N_k-1\}$. In addition, since
  $(X^n)_{n=0}^{N_k}$ satisfies \eqref{BDF2 scheme} for all $v \in V$, we
  obtain $\P$-almost surely with $v = 4X^n$ that
  \begin{align*}
    &4\Big( \frac{3}{2}X^n-2X^{n-1}+\frac{1}{2}X^{n-2}, X^n \Big)_H
    + 4k \langle A(X^n), X^n \rangle_{V^* \times V}\\
    &\qquad = 4\Big( \frac{3}{2}B(X^{n-1})\Delta_k W^n
    - \frac{1}{2}B(X^{n-2})\Delta_k W^{n-1},X^n \Big)_H
  \end{align*}
  for each $n = 2, \ldots, N_k$. By applying the identity \eqref{eqn:id 2-step},
  summing over $n$ from $2$ to $j$ and taking expectation, we see that
  \begin{equation}
  \label{eq:ap_1}
    \begin{aligned}
      & \E \big[\|X^j\|^2_H\big] + \E \big[\|2X^j-X^{j-1}\|^2_H\big]
      + \sum_{n=2}^j \E \big[\|X^n-2 X^{n-1}+X^{n-2}\|^2_H\big]\\
      &\quad = \E \big[\|X^1\|^2_H\big] + \E \big[\|2X^1-X^0\|^2_H\big]
      + 2 \sum_{n=2}^j (-2k) \E \big[\langle A(X^n),
      X^n \rangle_{V^* \times V}\big]\\
      &\qquad + 2 \sum_{n=2}^j \E \Big[\big( 3 B(X^{n-1})\Delta_k W^n
      - B(X^{n-2})\Delta_k W^{n-1},X^n \big)_H\Big].
    \end{aligned}
  \end{equation}
  An application of the coercivity condition \eqref{con:coer AB} shows that
  \begin{align}
    \label{eq3:Acoerc}
    \begin{split}
      -2k \E \big[\langle A(X^n), X^n \rangle_{V^* \times V}\big]
      &\leq k\kappa \E \big[\|X^n\|_H^2\big]
      - k \nu \E \big[\|B(X^n)\|^2_{\LB_2(U_0,H)}\big]\\
     &\quad - k\mu \E \big[\|X^n\|^p_V\big] + kc.
    \end{split}
  \end{align}
  After some elementary calculations, we obtain the following decomposition
  \begin{equation}
  \label{eq:ap_2}
    \begin{aligned}
    &\big( 3B(X^{n-1})\Delta_k W^n - B(X^{n-2})\Delta_k W^{n-1}, X^n \big)_H\\
    &\quad= \big( B(X^{n-1})\Delta_k W^n-B(X^{n-2})\Delta_k W^{n-1},
    X^n-2X^{n-1}+X^{n-2} \big)_H\\ 
    &\qquad + \big( B(X^{n-1})\Delta_k W^n, 2X^n-X^{n-1} \big)_H\\
    &\qquad - \big( B(X^{n-2})\Delta_k W^{n-1}, 2X^{n-1}-X^{n-2} \big)_H\\ 
    &\qquad + \big( B(X^{n-1})\Delta_k W^n, 3X^{n-1}-X^{n-2} \big)_H.
    \end{aligned}
  \end{equation}
  Since the random variables $(3X^{n-1}-X^{n-2})$ and $B(X^{n-2})
  \Delta_k W^{n-1}$ are $\F_{t_{n-1}}$-measurable and integrable for every $n
  \le j$, we use the
  martingale property of the stochastic integral to deduce
  \begin{equation*}
    \E \big[ \big( B(X^{n-1})\Delta_k W^n, 3X^{n-1}-X^{n-2} \big)_H \big] = 0
  \end{equation*}
  as well as
  \begin{equation*}
    \E \big[ \big( B(X^{n-1})\Delta_k W^n,
    B(X^{n-2}) \Delta_k W^{n-1} \big)_H \big] = 0.
  \end{equation*}
  By applying Young's inequality with weight $\nu$ to the decomposition
  \eqref{eq:ap_2} and taking expectation, we conclude that
  \begin{align*}
    &\E\big[\big( 3B(X^{n-1})\Delta_k W^n - B(X^{n-2})\Delta_k W^{n-1},
    X^n \big)_H\big]\\
    &\quad \leq \frac{\nu}{2} \Big( \E \big[\|B(X^{n-1})\Delta_k W^n\|_H^2\big]
    + \E \big[\|B(X^{n-2}) \Delta_k W^{n-1}\|_H^2\big] \Big)\\
    &\qquad + \frac{1}{2\nu} \E \big[\|X^n-2X^{n-1}+X^{n-2}\|_H^2\big]
    + \E \big[\big( B(X^{n-1})\Delta_k W^n,2X^n-X^{n-1} \big)_H\big]\\
    &\qquad - \E \big[\big( B(X^{n-2})\Delta_k W^{n-1},
    2X^{n-1}-X^{n-2} \big)_H\big].
  \end{align*}
  Inserting this and \eqref{eq3:Acoerc} into equation
  \eqref{eq:ap_1} then gives
  \begin{align*}
    & \E \big[\|X^j\|^2_H\big] + \E \big[\|2X^j-X^{j-1}\|^2_H\big]
    + \frac{\nu-1}{\nu}\sum_{n=2}^j \E \big[\|X^n-2 X^{n-1}+X^{n-2}\|^2_H\big]\\
    &\quad = \E \big[\|X^1\|^2_H\big] + \E \big[\|2X^1-X^0\|^2_H\big]
    + 2k\kappa \sum_{n=2}^j \E \big[\|X^n\|_H^2\big]\\
    &\qquad -2k\nu \sum_{n=2}^j \E \big[\|B(X^n)\|^2_{\LB_2(U_0,H)}\big]
    -2k\mu \sum_{n=2}^j \E \big[\|X^n\|^p_V\big] + 2kc (j-1)\\
    &\qquad + \nu \sum_{n=2}^j \Big( \E \big[\|B(X^{n-1})\Delta_k W^n\|_H^2\big]
    + \E \big[\|B(X^{n-2})\Delta_k W^{n-1}\|_H^2\big] \Big)\\
    &\qquad + 2 \E \big[\big( B(X^{j-1})\Delta_k W^j, 2X^j -X^{j-1} \big)_H\big]
    - 2 \E \big[\big( B(X^0)\Delta_k W^1, 2X^1-X^0 \big)_H\big].
  \end{align*}
  Applying again Young's inequality with weight $\nu$ and rearranging the terms
  yield
  \begin{align*}
    &(1-2k\kappa) \E \big[\|X^j\|^2_H\big]
    + 2k\nu \E \big[\|B(X^j)\|^2_{\LB_2(U_0,H)}\big]
    + 2k\mu \sum_{n=2}^j \E \big[\|X^n\|^p_V\big] \\
    &\qquad + \frac{\nu-1}{\nu} \Big( \E \big[\|2X^j-X^{j-1}\|^2_H\big]
    + \sum_{n=2}^j \E \big[\|X^n-2 X^{n-1}+X^{n-2}\|^2_H\big] \Big)\\
    &\quad \leq \E \big[\|X^1\|^2_H\big] 
    + \frac{\nu+1}{\nu} \E \big[\|2X^1-X^0\|^2_H\big]
    + 2k\kappa \sum_{n=2}^{j-1} \E \big[\|X^n\|_H^2\big] + 2Tc \\
    &\qquad -2k\nu \sum_{n=2}^{j-1} \E \big[\|B(X^n)\|^2_{\LB_2(U_0,H)}\big]
    + 2\nu \sum_{n=1}^j \E \big[\|B(X^{n-1})\Delta_k W^n\|_H^2\big].
  \end{align*}
  Next, due to the It\=o isometry the last two sums on the right-hand side
  almost cancel each other up to two summands.
  Moreover, a further application of Young's inequality yields
  \begin{align*}
    \|2X^1-X^0\|_H^2 = 4 \|X^1\|_H^2 - 4 (X_1, X_0)_H + \|X^0\|_H^2
    \leq 5 \big(\|X^1\|_H^2 + \|X^0\|_H^2).
  \end{align*}
  Since $1 > 1-2k\kappa > 0$  and $\frac{\nu+1}{\nu} \le 2$ by assumption, we
  obtain 
  \begin{align*}
    &\E \big[\|X^j\|^2_H\big]
    + 2k\nu \E \big[\|B(X^j)\|^2_{\LB_2(U_0,H)}\big]
    + 2k\mu \sum_{n=2}^j \E \big[\|X^n\|^p_V\big] \\
    &\qquad + \frac{\nu-1}{\nu} \Big( \E \big[\|2X^j-X^{j-1}\|^2_H\big]
    + \sum_{n=2}^j \E \big[\|X^n-2 X^{n-1}+X^{n-2}\|^2_H\big] \Big)\\
    &\quad \leq \frac{2k\kappa}{1-2k\kappa} \sum_{n=2}^{j-1}
    \E \big[\|X^n\|_H^2\big] + \frac{2Tc}{1-2k\kappa}\\
    &\qquad + \frac{1}{1-2k\kappa}
    \Big(11 \sum_{n=0}^{1} \E \big[\|X^n\|_H^2\big]
    + 2k\nu \sum_{n=0}^{1} \E \big[\|B(X^n)\|^2_{\LB_2(U_0,H)}\big] \Big).
  \end{align*}
  Applying a discrete version of Gronwall's inequality, see, e.g.,
  \cite{clark1987}, yields the estimate \eqref{ineq:ap} and hence the result.
\end{proof}

Under the assumptions stated in this section, the existence and uniqueness of a
solution to implicit methods such as the BDF2-Mayurama scheme \eqref{BDF2 scheme}
and the BEM scheme \eqref{BEM scheme} can be proven through techniques from
nonlinear PDE theory. These techniques rely on the monotonicity condition
\eqref{con:mono AB} and have been used to show well-posedness of the one-step
BEM scheme applied to nonlinear stochastic evolution equations, see, e.g.,
\cite[Theorem~3.3]{emmrich2017} or \cite[Theorem~2.9]{gyoengy2005}. Here, we
adapt this approach to the multi-step BDF2-Mayurama scheme in order to prove
well-posedness.

\begin{theorem}
\label{thm:scheme}
  Let Assumption~\ref{assump AB1} be satisfied for some $p \in (1,\infty)$. Let
  $h \in (0,1)$ and $k = \frac{T}{N_k}$, $N_k \in \N$, be fixed with
  $k\kappa \leq 3$ and let some initial values $(X^n_{k,h})_{n=0,1}$ satisfy
  Assumption~\ref{assump IC}. Then the numerical scheme \eqref{BDF2 scheme} has
  a unique $(\F_{t_n})_{n=0}^{N_k}$-adapted and $\P$-almost surely $V_h$-valued
  solution $(X_{k,h}^n)_{n=0}^{N_k}$. In addition, if $2k\kappa < 1$ holds, then
  the random variables $X_{k,h}^n$ are
  $L^p(\Omega;V) \cap L^2(\Omega;H)$-integrable for every
  $n \in \{2,\ldots,N_k\}$.
\end{theorem}

The following lemmas are needed to show existence as well as adaptedness of
a discrete solution to the numerical scheme \eqref{BDF2 scheme}. A proof of
each result can be found, respectively, in \cite[Section~9.1]{evans1998} and
\cite[Lemma~4.3]{eisenmann2019}.

\begin{lemma}
\label{lem:zero}
  Let $R \in (0,\infty)$ and let $f \colon \R^N \to \R^N$, $N \in \N$, be a
  continuous function. If $f(x) \cdot x \geq 0$ holds for every $x \in \R^N$
  with $\|x\|_2 = R$, where $\|\cdot\|_2$ denotes the Euclidean norm, then
  there exists $x_0 \in \R^N$ with $\|x_0\|_2 \leq R$ satisfying $f(x_0) = 0$.
\end{lemma}

\begin{lemma}
\label{lem:meas_zero}	
  Let $(\Omega,\F,(\F_t)_{t \in [0,T]},\P)$ be a filtered probability space.
  Let $\F_t$ be a complete sub-$\sigma$-algebra  of $\F$ for fixed
  $t \in [0,T]$ and let $\Omega' \in \F_t$ with $\P(\Omega') = 1$. Further,
  let the function
  $f \colon \Omega \times \R^N \to \R^N$, $N \in \N$, be $\F_t$-measurable
  in the first argument for every $x \in \R^N$ and continuous in the second
  argument for every $\omega \in \Omega'$. Moreover, assume for each
  $\omega \in \Omega'$ that the equation $f(\omega,x) = 0$ has a unique
  solution $x(\omega) \in \R^N$. Then the mapping
  \begin{equation*}
    x \colon \Omega \to \R^N, \quad \omega \mapsto
    \begin{cases}
      x(\omega) &\text{for } \omega \in \Omega',\\
      0 &\text{otherwise},
    \end{cases}
  \end{equation*}
  is $\F_t$-measurable.
\end{lemma}

\begin{proof}[Proof of Theorem~\ref{thm:scheme}]
  First, we will show by an inductive argument over $n \in \{1,\ldots,N_k\}$
  the existence and the $\P$-almost sure uniqueness of random variables $X^n$
  which solve \eqref{BDF2 scheme} and are $\F_{t_n}$-measurable as well as
  almost surely $V_h$-valued. Notice that this assertion follows for $n=1$ from
  Assumption~\ref{assump IC}. Hence we assume that $\F_{t_j}$-measurable and
  $\P$-almost surely $V_h$-valued random variables $X^j$ satisfying
  \eqref{BDF2 scheme} exist for $j = 0, \dots, n-1$, $n \ge 2$.
	
  Let $(\phi_i)^{N_h}_{i = 1}$ be a basis of the $N_h$-dimensional subspace
  $V_h \subset V$. We will identify uniquely every $X \in V_h$ with a vector
  $\boldsymbol{X} = (\boldsymbol{X}_1, \ldots, \boldsymbol{X}_{N_h})^T \in
  \R^{N_h}$ by the relation $X = \sum_{i=1}^{N_h} \boldsymbol{X}_i \phi_i$ and
  define a norm on $\R^{N_h}$ by
  $\|\boldsymbol{X}\|_{\R^{N_h}} \coloneqq \|X\|_H$. Since the filtered
  probability space satisfies the usual conditions, $\F_0$ contains all
  $\P$-null sets. Hence, we can choose $\Omega' \in \F_{t_n}$ with
  $\P(\Omega') = 1$ such that the Wiener process $W$ is $U$-valued and
  the random variables $X^0, \ldots, X^{n-1}$ are $V_h$-valued on $\Omega'$.
  
  For any $\omega \in \Omega'$ and $\boldsymbol{X} \in \R^{N_h}$ with
  associated $X \in V_h$ we define the function
  $f \colon \Omega \times \R^{N_h} \rightarrow \R^{N_h}$ componentwise for
  $i=1,\ldots,{N_h}$ by
  \begin{align*}
    &f(\omega,\boldsymbol{X})_i \coloneqq \big( 3 X- 4X^{n-1}(\omega)
    + X^{n-2}(\omega),\phi_i \big)_H
    + 2 k \langle A(X,\omega), \phi_i \rangle_{V^* \times V}\\
    &\qquad - \big( 3 B(X^{n-1}(\omega),\omega)\Delta_k W^n(\omega)
    - B(X^{n-2}(\omega),\omega)\Delta_k W^{n-1}(\omega),
    \phi_i \big)_H.
  \end{align*}
  Notice that $\boldsymbol{X}$ fulfills the equation
  $f(\omega,\boldsymbol{X}) = 0$ if and only if $X^n(\omega) \coloneqq X=
  \sum_{i=1}^{N_h} \boldsymbol{X}_i \phi_i$ 
  solves the equation \eqref{BDF2 scheme} for given $\omega \in \Omega'$.
  
  In the following, let $\omega \in \Omega'$ be arbitrary, but fixed. 
  To prove the existence of a zero of the function $f(\omega,\cdot)$,
  we will show that the mapping
  $f(\omega,\cdot) \colon \R^{N_h} \to \R^{N_h}$ is continuous and satisfies
  $f(\omega,\boldsymbol{X}) \cdot \boldsymbol{X} \geq 0$ for some
  $R \in (0,\infty)$ and all $\boldsymbol{X} \in \R^{N_h}$ with
  $\|\boldsymbol{X}\|_2 = R$. The hemicontinuity of $A$ and the monotonicitiy
  condition \eqref{con:mono AB} imply the demicontinuity of the operator
  $A(\cdot,\omega)$. Since weak and strong convergence are equivalent in
  finite-dimensional spaces, the function $f(\omega,\cdot)$ is continuous.
  Moreover, we observe that
  \begin{equation*}
    \begin{aligned}
      &f(\omega,\boldsymbol{X}) \cdot \boldsymbol{X}
      = \big( 3 X- 4X^{n-1}(\omega) + X^{n-2}(\omega), X \big)_H
      + 2 k \langle A(X,\omega), X \rangle_{V^* \times V}\\
      &\qquad - \big( 3 B(X^{n-1}(\omega),\omega)\Delta_k W^n(\omega)
      - B(X^{n-2}(\omega),\omega) \Delta_k W^{n-1}(\omega), X \big)_H
    \end{aligned}
  \end{equation*}
  holds for every $\boldsymbol{X} \in \R^{N_h}$. Applying the Cauchy--Schwarz
  inequality and the coercivity condition \eqref{con:coer AB} leads to
  \begin{equation*}
    \begin{aligned}
    &f(\omega,\boldsymbol{X}) \cdot \boldsymbol{X}
    \geq (3-k \kappa) \|X\|_H^2
    - \|4X^{n-1}(\omega)-X^{n-2}(\omega)\|_H \|X\|_H\\
    &\qquad + \nu k \|B(X,\omega)\|_{\LB_2(U_0,H)}^2
    + \mu k \|X\|_V^p - c k\\
    &\qquad - \|3B(X^{n-1}(\omega),\omega)\Delta_k W^n(\omega)
    - B(X^{n-2}(\omega),\omega)\Delta_k W^{n-1}(\omega)\|_H \|X\|_H.
    \end{aligned}
  \end{equation*}
  Using the assumption $3-k \kappa \geq 0$ and the continuity of the embedding
  $V \hookrightarrow H$, we derive the estimate
  \begin{align*}
    &f(\omega,\boldsymbol{X}) \cdot \boldsymbol{X} \geq \|X\|_H \Big(
    \beta^p_{V \hookrightarrow H} \mu k \|X\|_H^{p-1}
	- \|4X^{n-1}(\omega)-X^{n-2}(\omega)\|_H\\
	&\qquad - \|3B(X^{n-1}(\omega),\omega)\Delta_k W^n(\omega)
	- B(X^{n-2}(\omega),\omega)\Delta_k W^{n-1}(\omega)\|_H \Big)
	- ck.
  \end{align*}
  Since $\|X\|_H = \|\boldsymbol{X}\|_{\R^{N_h}}$ and norms on the
  finite-dimensional space $\R^{N_h}$ are equivalent, there is some constant
  $C \in (0,\infty)$ such that
  \begin{align*}
  &f(\omega,\boldsymbol{X}) \cdot \boldsymbol{X}
  \geq C\|\boldsymbol{X}\|_2 \Big(
  \beta^p_{V \hookrightarrow H} \mu k \|\boldsymbol{X}\|_2^{p-1}
  - \|4X^{n-1}(\omega)-X^{n-2}(\omega)\|_H\\
  &\qquad - \|3B(X^{n-1}(\omega),\omega)\Delta_k W^n(\omega)
  - B(X^{n-2}(\omega),\omega)\Delta_k W^{n-1}(\omega)\|_H \Big)
  - ck.
  \end{align*}
  Now we choose $R(\omega) \in (0,\infty)$ sufficiently large such that
  $f(\omega,\boldsymbol{X}) \cdot \boldsymbol{X} \geq 0$ holds for all
  $\boldsymbol{X} \in \R^{N_h}$ with
  $\|\boldsymbol{X}\|_{2} = R(\omega)$. From Lemma~\ref{lem:zero}
  it then follows that a zero of the function $f(\omega,\cdot)$ exists.
  
  To prove the uniqueness of a zero of the function $f(\omega,\cdot)$ for fixed
  $\omega \in \Omega'$, assume that two distinct solutions
  $\boldsymbol{X},\boldsymbol{Y} \in \R^{N_h}$ with associated $X,Y \in V_h$,
  respectively, exist such that
  $f(\omega,\boldsymbol{X}) = f(\omega,\boldsymbol{X}) = 0$. The monotonicity
  condition \eqref{con:mono AB} and the condition $k \kappa \leq 3$ imply that
  \begin{align*}
    0 &= (f(\omega,\boldsymbol{X})-f(\omega,\boldsymbol{Y}),
    \boldsymbol{X}-\boldsymbol{Y})_2\\
    &= 3 \|X-Y\|_H^2 + k \langle A(X,\omega)-A(Y,\omega),
    X-Y \rangle_{V^* \times V}
    \geq (3-k\kappa) \|X-Y\|_H^2 \geq 0.
  \end{align*}
  This shows that $X$ and $Y$ coincide in $V_h$ and hence $\boldsymbol{X}$ and
  $\boldsymbol{Y}$ coincide in $\R^{N_h}$. Therefore, $f(\omega,\cdot)$ has a
  unique zero for every $\omega \in \Omega'$.
  
  Now, we set $X^n(\omega) \coloneqq X \in V_h$ for every $\omega \in \Omega'$
  and $X^n(\omega) \coloneqq 0$ for each $\omega \in \Omega \setminus \Omega'$.
  To prove the $\F_{t_n}$-measurability of $X^n$, recall that $X^j(\cdot)$ is
  assumed to be measurable with respect to $\F_{t_j} \subset \F_{t_n}$ for each
  $j=0,\ldots,n-1$. Moreover, Assumption~\ref{assump AB1} and the measurability
  properties of the Wiener process $W$ imply the $\F_{t_n}$-measurability of
  $A(v,\cdot)$ and $B(X^j(\cdot),\cdot)\Delta_k W^n(\cdot)$ for every
  $v \in V_h$ and $j =0,\ldots,n-1$. Therefore, the function
  $f(\cdot,\boldsymbol{X})$ is $\F_{t_n}$-measurable for every fixed
  $\boldsymbol{X} \in \R^{N_h}$. Since the $\sigma$-algebra $\F_{t_n}$
  contains all $\P$-null sets, we deduce from Lemma~\ref{lem:meas_zero} the
  measurability of the mapping $\omega \mapsto X^n (\omega)$ with respect to
  $\F_{t_n}$.
  
  This concludes the proof for the existence and uniqueness of a solution
  $(X^n)_{n=0}^{N_k}$ to the numerical scheme \eqref{BDF2 scheme}. The
  $L^p(\Omega;V) \cap L^2(\Omega;H)$-integrability of the discrete solution
  follows for sufficiently small temporal step size $k$ from
  Theorem~\ref{thm:a priori}.
\end{proof}

\begin{remark}
  \label{rmk:scheme BEM}
  Consider the BEM scheme \eqref{BEM scheme} with an initial value $X^0_{k,h}$
  satisfying
  \begin{equation*}
    X^0_{k,h} \in L^2(\Omega,\F_{t_0},\P;H) \quad \mathrm{ and } \quad
    B(X^0_{k,h}) \in L^2(\Omega, \F_{t_0}, \P; \LB_2(U_0,H))
  \end{equation*}
  such that $X^0_{k,h}$ is $\P$-almost surely $V_h$-valued. Under
  Assumption~\ref{assump AB1}, the BEM scheme admits for every temporal
  step size $k = \frac{T}{N_k}$ with $k\kappa<1$ a unique solution
  $(X_{k,h}^n)_{n=0}^{N_k}$, which is $(\F_{t_n})_{n=0}^{N_k}$-adapted,
  $\P$-almost surely $V_h$-valued and
  $L^p(\Omega;V) \cap L^2(\Omega;H)$-integrable. This result can be proven
  for $p \in (1,\infty)$ with similar techniques as used in the proof of
  Theorem~\ref{thm:scheme}. In the case of $p \in [2,\infty)$, an
  alternative proof can be found in \cite{gyoengy2005}.
\end{remark}

\begin{remark}
  \label{rmk:ini val BDF2}
  Let $h\in (0,1)$ and $V_h \neq \{0\}$ be fixed. The initialization of the
  BDF2-Maruyama scheme \eqref{BDF2 scheme} requires two initial values
  $(X^n_{k,h})_{n=0,1}$ which are $(\F_{t_n})_{n=0,1}$-adapted and
  $\P$-almost surely $V_h$-valued. A typical choice for the first initial value
  is $X_{k,h}^0 = P_h(X_0)$, where $P_h \colon H \to V_h$ denotes
  the orthogonal projector onto $V_h$ with respect to the inner product in $H$.
  In Section~\ref{sec:simu}, we also consider an interpolation operator
  as an alternative to $P_h$. 
  Further, one iteration of the BEM scheme \eqref{BEM scheme} with
  $X_{k,h}^0$ as the initial value yields an $\F_{t_1}$-measurable
  and $\P$-almost surely $V_h$-valued random variable $X^1_{k,h}$
  which is an admissible choice for the second initial value.
  Compare further with Remark~\ref{rmk:scheme BEM}.
\end{remark}

We close this section with a simple example of a stochastic partial
differential equation, which fits into the framework of
Assumption~\ref{assump AB1}. For further examples of stochastic evolution
equations we refer to \cite[Section~4.1]{prevot2007} and Section~\ref{sec:simu}
below.

\begin{example}
  \label{ex:SHE}
  We consider the stochastic heat equation
  \begin{equation}
  \label{prob:SHE}
    \begin{aligned}
      \diff u(t,x) - u_{xx}(t,x) \diff t &= \sigma \diff W(t,x),
      && (t,x) \in (0,T] \times (0,1),\\
      u(t,0) &= u(t,1) = 0, &&t \in (0,T],\\
      u(0,x) &= \sin(\pi x), &&x \in (0,1),
    \end{aligned}
  \end{equation}
  with additive noise determined by the scalar $\sigma \in \R$, Dirichlet
  boundary conditions and a smooth deterministic initial value.

  In the context of our abstract setting, we make use of the Gelfand triple
  induced by the spaces $V = H^1_0(0,1)$ and $H = L^2(0,1)$ and identify $X$ as
  the abstract function of $u$ such that
  \begin{equation*}
    X \colon [0,T] \times \Omega \to V,
    \quad (t,\omega) \mapsto u(t,\cdot,\omega).
  \end{equation*}
  The deterministic initial value $X_0 = \sin(\pi \cdot)$ is smooth and
  equal to zero on the boundary. Hence, we have $X_0 \in V$.
  The Wiener process $W$ is assumed to take values in $U =
  L^2(0,1)$ and its covariance operator $Q$ to have finite trace. If
  $\{\chi_j\}_{j \in N}$ is an orthonormal basis of $U$ consisting of
  eigenfunctions of $Q$ with eigenvalues $q_j \geq 0$, then
  $\{Q^{\frac{1}{2}}\chi_j\}_{j \in N}$ is an orthonormal basis of
  $U_0 = Q^{\frac{1}{2}}(U) \subset H$ and it follows
  \begin{equation*}
    \|\sigma \, \id_H\|_{\LB_2(U_0,H)}^2
    = \sigma^2  \sum_{j \in \N} \|Q^\frac{1}{2}\chi_j\|_H^2
    = \sigma^2 \sum_{j \in \N} q_j
    = \sigma^2 \Tr(Q) < \infty.
  \end{equation*}
  In addition, the operators
  \begin{align*}
    &A \colon V \to V^\ast, \quad v \mapsto A(v) \\
    &B \colon V \to \LB_2(U_0,H), \quad v \mapsto \sigma \, \id_H
  \end{align*}
  are well-defined, where $A(v) \in V^\ast = H^{-1}(0,1)$ is the linear
  functional given by
  \begin{align*}
    \langle A(v),w \rangle_{V^* \times V} = \int_0^1 v'(x) w'(x) \diff{x}
    = (v,w)_{V}
  \end{align*}
  for all $v,w \in V = H^1_0(\Omega)$. Altogether, this allows us to 
  reformulate problem \eqref{prob:SHE} as a stochastic evolution equation of
  the form \eqref{sde_problem}.  

  The operators $A$ and $B$ are both deterministic and $\B(V)$-measurable.
  In particular, the linear operator $A \colon V \to V^\ast$ is
  hemi-continuous, bounded and, hence, of linear growth.
  The monotonicity condition \eqref{con:mono AB} and the coercivity condition
  \eqref{con:coer AB} are satisfied with $\kappa=0$, $\mu=1$, $p = 2$, and $c
  =\nu \sigma^2 \Tr(Q)$ for any $\nu \in [1,\infty)$. 

  For fixed $h>0$ and some finite-dimensional subspace
  $\{0\} \neq V_h \subset V$, we generate the initial values for the
  BDF2-Maruyama scheme as discussed in Remark~\ref{rmk:ini val BDF2}. Further,
  the initial values $(X^n_{k,h})_{n=0,1}$ are $L^2(\Omega;H)$-integrable by
  construction. Since the operator $B$ is constant, the terms
  $B(X^n_{k,h})$, $n \in \{0,1\}$, also fulfill the
  integrability condition in Assumption~\ref{assump IC}.
  
  Consequently, Assumption~\ref{assump AB1} and Assumption~\ref{assump IC} are
  satisfied and Theorem~\ref{thm:scheme} guarantees for every sufficiently
  small step size $k$ that the BDF2-Maruyama scheme is well-defined for the
  given problem \eqref{prob:SHE}.
\end{example}

\section{Convergence of the BDF2-Maruyama method}
\label{sec:converg}

In this section, we derive an estimate for the strong error between 
the approximate solution $(X^n_{k,h})_{n=0}^{N_k}$ of \eqref{BDF2 scheme} and
the exact solution $X$ of \eqref{sde_problem}. In order to determine a lower
bound for the order of convergence, we have to impose additional conditions on
the operators $A$ and $B$ for the error analysis.

Throughout this section, we fix $p=2$ and a Gelfand triple $(V,H,V^\ast)$ with
$V$ being uniformly convex as discussed in Section~\ref{sec:prelimi}.

\begin{assumption}
  \label{assump AB2}
  Let the operators $A$ and $B$ satisfy Assumption~\ref{assump AB1} for $p=2$.
  Moreover, there are $\kappa \in [0,\infty)$, $\nu \in (1,\infty)$ and
  $L, K \in (0,\infty)$ such that the operators $A$ and $B$ satisfy
  $\P$-almost surely on $\Omega$ for all $v,u \in V$ the monotonicity condition
  \begin{equation}
    \label{con:str mono AB}
    \begin{aligned}
      &2 \langle A(u)-A(v),u-v \rangle_{V^* \times V} + \kappa \|u-v\|_H^2\\
      &\qquad \geq \nu \|B(u)-B(v)\|_{\LB_2(U_0,H)}^2 + K \|u-v\|_V^2
    \end{aligned}
  \end{equation}
  and the Lipschitz condition
  \begin{equation}
    \label{con:Lip A}
    \|A(v)-A(u)\|_{V^*} \leq L \|v-u\|_V.
  \end{equation}
\end{assumption}

In order to determine an order of convergence, we require the consistency
of the initial values for the numerical method \eqref{BDF2 scheme}.

\begin{assumption}
  \label{assump IC2}
  Let the initial values $(X^n_{k,h})_{n=0,1}$ satisfy
  Assumption~\ref{assump IC}. In addition, there exist $C_I \in (0,\infty)$ and
  $\gamma \in (0,\infty)$ such that
  \begin{equation*}
    \sum_{n = 0}^1 \E \big[\|X^n_{k,h} - X(t_n)\|_H^2
    + k \|B(X^n_{k,h})-B(X(t_n))\|^2_{\LB_2(U_0,H)}\big]
    \le C_I ( k + h^\gamma)
  \end{equation*}
  holds for all $k = \frac{T}{N_k}$, $N_k \in \N$, and $h \in (0,1)$.
\end{assumption}

We also need to impose the following additional temporal regularity condition
on the exact solution. To this end, we recall the definition of the
$\p$-variation norm from \eqref{def:p_var}.

\begin{assumption}
  \label{assump X}  
  Let the initial value $X_0$ be $V$-valued $\P$-almost surely and let the
  solution to \eqref{sde_problem} satisfy
  $X \in C([0,T]; L^2(\Omega;V))$. In addition, there exists
  $\p \in [2,\infty)$ with $\|X\|_{\p-\mathrm{var},L^2(\Omega;V)} < \infty$,
  i.e. $X$ is of finite $\p$-variation with respect to the
  $L^2(\Omega;V)$-norm.
\end{assumption}

Evidently, if the exact solution $X \in C([0,T]; L^2(\Omega;V))$ is H\"older
continuous with exponent $\gamma = \frac{1}{q} \in (0,\frac{1}{2})$ then 
Assumption~\ref{assump X} is satisfied. The following two lemmas show how the
$\p$-variation norm is applied in the error analysis.  

\begin{lemma}
  \label{lem:var_norm}
  Let $X \in C([0,T];L^2(\Omega;V))$ be a stochastic process of finite
  $\p$-varia\-tion with respect to the $L^2(\Omega;V)$-norm for some $\p \in
  [2,\infty)$. Then it holds for every finite partition $\mathcal{P}
  = \{t_0=0,\ldots,t_N=T\}$,
  $N \in \N$, of the interval $[0,T]$ with maximal step size
  $k \coloneqq \max_{n=1,\ldots,N} (t_n-t_{n-1})$ that
  \begin{equation*}
    \sum_{n=1}^N \int_{t_{n-1}}^{t_n} \E \big[ \|X(s)-X(t_n)\|^2_V \big] \diff s
    \leq (1+T) k^{\frac{2}{\p}} \|X\|^2_{\p-\mathrm{var},L^2(\Omega;V)}.
  \end{equation*}
\end{lemma} 

\begin{proof}
  The assumption $X \in C([0,T];L^2(\Omega;V))$ implies that the real-valued
  function $s \mapsto \E \big[ \|X(s)-X(t_n)\|_V^2 \big]$ is continuous on the
  interval $[0,T]$. By the intermediate value theorem, there exist for each
  $n \in \{1,\ldots,N\}$ a point $\xi_n \in [t_{n-1},t_n]$ independent of
  $\Omega$ such that
  \begin{equation*}
    \int_{t_{n-1}}^{t_n} \E \big[ \|X(s)-X(t_n)\|^2_V \big] \diff s
    = (t_n-t_{n-1}) \|X(\xi_n)-X(t_n)\|^2_{L^2(\Omega;V)}.
  \end{equation*}
  A summation over $n$ from $1$ to $N$ shows that
  \begin{equation*}
    \sum_{n=1}^N \int_{t_{n-1}}^{t_n}
    \E \big[ \|X(s)-X(t_n)\|^2_V \big] \diff s
    = \sum_{n=1}^N (t_n-t_{n-1})^{\frac{\p-2}{\p}+\frac{2}{\p}}
    \|X(\xi_n)-X(t_n)\|^2_{L^2(\Omega;V)}.
  \end{equation*}
  Finally, applying H\"older's inequality with exponents
  $\rho = \frac{\p}{\p-2}$ and $\rho' = \frac{\p}{2}$ yields
  \begin{align*}
    &\sum_{n=1}^N \int_{t_{n-1}}^{t_n} \E \big[ \|X(s)-X(t_n)\|^2_V \big]
    \diff s\\
    &\quad \leq \Big( \sum_{n=1}^N (t_n-t_{n-1}) \Big)^{\frac{\p-2}{\p}}
    \Big( \sum_{n=1}^N (t_n-t_{n-1}) \|X(\xi_n)-X(t_n)\|^\p_{L^2(\Omega;V)}
    \Big)^{\frac{2}{\p}}\\
    &\quad \leq (1+T) k^{\frac{2}{\p}}
    \|X\|^2_{\p-\mathrm{var},L^2(\Omega;V)},
  \end{align*}
  where we use $T^{\frac{q-2}{q}} \leq(1 + T)$ for $T \in (0,\infty)$ and
  recall the definition of the $\p$-variational norm from \eqref{def:p_var}
  in the last step.
\end{proof}

\begin{lemma}
  \label{lem:est_int_A}
  Let Assumption~\ref{assump AB2} and Assumption~\ref{assump X} be satisfied
  with $\p \in [2,\infty)$. Then, it holds for every finite partition
  $\mathcal{P} = \{t_0=0,\ldots,t_N=T\}$, $N \in \N$, of the interval $[0,T]$
  with maximal step size $k \coloneqq \max_{n=1,\ldots,N} (t_n-t_{n-1})$ that
  \begin{equation*}
    \sum_{n=1}^N \int_{t_{n-1}}^{t_n}
    \E \big[ \|A(X(s))-A(X(t_n))\|^2_{V^*} \big] \diff s
    \leq L_A k^{\frac{2}{\p}}
    \|X\|^2_{\p-\mathrm{var},L^2(\Omega;V)}
  \end{equation*}
  with constant $L_A = (1+T) L^2$.
\end{lemma}

\begin{proof}
  The Lipschitz condition \eqref{con:Lip A} yields for every
  $n \in \{1,\ldots,N\}$ and $s \in [t_{n-1},t_n]$ that
  \begin{equation*}
	\E \big[ \|A(X(s))-A(X(t_n))\|^2_{V^*} \big]
	\leq L^2 \, \E \big[ \|X(s)-X(t_n)\|_V^2 \big].
  \end{equation*}
  The assertion follows now from an application of Lemma~\ref{lem:var_norm}.
\end{proof}

\begin{lemma}
  \label{lem:est_stoch_int}
  Let Assumption~\ref{assump AB2} and Assumption~\ref{assump X} with $\p \in
  [2,\infty)$ be satisfied. Then, it holds
  for every finite partition $\mathcal{P} = \{t_0=0,\ldots,t_N=T\}$,
  $N \in \N$, of the interval $[0,T]$ with maximal step size
  $k \coloneqq \max_{n=1,\ldots,N} (t_n-t_{n-1})$ that
  \begin{align*}
    \sum_{n = 1}^{N} \E \Big[ \Big\|
    \int_{t_{n-1}}^{t_n} B(X(s)) - B(X(t_{n-1})) \diff W(s)
    \Big\|_H^2 \Big]
    \le L_B k^{\frac{2}{\p}} \|X\|^2_{\p-\mathrm{var},L^2(\Omega;V)}
  \end{align*}
  with constant $L_B = (1+T)(\kappa \, \beta_{V \hookrightarrow H}^2 + 2L - K )
  \ge 0$.
\end{lemma}

\begin{proof}
  First, observe that the monotonicity-like condition \eqref{con:str mono AB}
  with $\nu \in (1,\infty)$ and the Lipschitz continuity of $A$ imply for every
  $u,v \in V$ that  
  \begin{align*}
    \|B(u)-B(v)\|_{\LB_2(U_0,H)}^2
    & \leq  2 \|A(u)-A(v)\|_{V^*} \|u-v\|_V
    + \kappa \|u-v\|_H^2- K \|u-v\|_V^2 \\ 
    &\leq \kappa \|u-v\|_H^2 + (2L - K) \|u-v\|_V^2\\
    &\leq (\kappa \, \beta_{V \hookrightarrow H}^2 + 2L - K )
    \|u-v\|_V^2\\
    &= \frac{L_B}{1+T} \|u-v\|_V^2 \quad \P\text{-a.s.,}
  \end{align*}
  where we also used that $V$ is continuously embedded into $H$, i.e.
  for all $v \in V$ it holds $\|v\|_H \leq \beta_{V \hookrightarrow H}
  \|v\|_V$. In particular, it also follows from the above estimate that $L_B
  \ge 0$.
 
  Next, for every $n \in \{1,\ldots,N\}$ an application
  of the It\=o isometry yields
  \begin{align*}
    &\E \bigg[ \Big\| \int_{t_{n-1}}^{t_n} B(X(s)) - B(X(t_{n-1})) \diff W(s)
    \Big\|_H^2\bigg]\\
    &\quad = \E \Big[\int_{t_{n-1}}^{t_n}
    \big\| B(X(s)) - B(X(t_{n-1})) \big\|_{\LB_2(U_0;H)}^2 \diff s\Big]\\
    &\quad \leq \frac{L_B}{1+T}
    \int_{t_{n-1}}^{t_n} \E \big[ \| X(s) - X(t_{n-1}) \|_V^2 \big] \diff{s}.
  \end{align*}
  Then, the assertion follows from an application of a slightly modified
  version of Lemma~\ref{lem:var_norm}.
\end{proof}

We are now prepared to state the main result of this section.

\begin{theorem}
  \label{thm:err}
  Let Assumption~\ref{assump AB2} and Assumption~\ref{assump X} be satisfied
  with $\p \in [2,\infty)$. Let $h \in (0,1)$ and $k = \frac{T}{N_k}$,
  $N_k \in \N$, be fixed with $2\kappa k < 1$. Further, let some initial values
  $(X^n_{k,h})_{n=0,1}$ satisfy Assumption~\ref{assump IC2} for
  $\gamma \in (0,\infty)$. Then, it holds the error estimate
  \begin{align*}  
    &\max_{n\in\{2,\ldots,N_k\}} \|X^n_{k,h}-X(t_n)\|_{L^2(\Omega;H)}^2
    + k \sum_{n=2}^{N_k} \|X^n_{k,h}-X(t_n)\|_{L^2(\Omega;V)}^2\\
    &\quad \leq C_k \Big( k^{\frac{2}{\p}}
    \|X\|^2_{\p-\mathrm{var},L^2(\Omega;V)}
    + k + h^\gamma
    + \max_{n\in\{2,\ldots,N_k\}} \E\big[\mathrm{dist}_H(X(t_n),V_h)^2\big]\\
    &\qquad + k \big(1+\|P_h\|_{\LB(V)}\big)^2
    \sum_{n=2}^{N_k} \E\big[\mathrm{dist}_V(X(t_n),V_h)^2\big] \Big),
  \end{align*}
  where the constant $C_k$ is defined for
  $\tilde{C}_k = \min\{1-2k\kappa,K\}$ by
  \begin{align*}
    C_k = 2\ee^{2T\kappa \tilde{C}_k^{-1}} \tilde{C}_k^{-1}
    \max\Big\{ 1, 8\frac{L^2}{K}+K, 16 C_I, 2\nu C_I,
    32\frac{L_A}{K} + \frac{2\nu L_B}{\nu-1} \Big\}.
  \end{align*}
\end{theorem}

\begin{proof}
  In the following, all equalities and inequalities involving random variables
  are assumed to hold $\P$-almost surely, unless stated otherwise.
  For $n=0,\ldots,N_k$, we denote the error of the discretization scheme
  \eqref{BDF2 scheme} at time $t_n$ by $E^n \coloneqq  X^n-X(t_n)$. Using
  the orthogonal projection $P_h \colon H \rightarrow V_h$, we split the
  error into two parts by writing 
  \begin{align*}
    E^n &= P_h E^n + (\id-P_h)E^n\\
    &= \left(X^n-P_h X(t_n)\right) + (P_h-\id)X(t_n)
    \eqqcolon \Theta^n + \Xi^n, \quad n=0,\ldots,N_k. 
  \end{align*}
  By definition, $\Theta^n$ and $\Xi^n$ are orthogonal with respect to the
  inner product $(\cdot,\cdot)_H$ and, hence, 
  \begin{equation}
  \label{eq:E_n H-norm}
    \|E^n\|_H^2 = \|\Theta^n\|_H^2 + \|\Xi^n\|_H^2.
  \end{equation}
  Let us fix $n \in \{2,\ldots,N_k\}$ for now. Recalling the identity 
  \eqref{eqn:id 2-step}, it holds that 
  \begin{equation}
  \label{eq:Gamma_n}
    \begin{aligned}
      & \| E^n \|^2_H - \|E^{n-1}\|^2_H +
      \|2 E^n- E^{n-1}\|^2_H - \|2 E^{n-1}- E^{n-2}\|^2_H\\ 
      &\quad + \| E^n- 2 E^{n-1} + E^{n-2}\|^2_H 
      = 4 \Big( \frac{3}{2} E^n- 2E^{n-1}
      + \frac{1}{2} E^{n-2}, E^n \Big)_H \eqqcolon \Gamma^n. 
    \end{aligned}
  \end{equation}
  We insert the $H$-orthogonal decomposition
  $E^n = \Theta^n + \Xi^n$ to obtain
  \begin{align}
    \label{eq:Gamma_n2}
    \begin{split}
      \Gamma^n &= 4 \Big( \frac{3}{2} E^n- 2E^{n-1} + \frac{1}{2}
      E^{n-2}, \Theta^n + \Xi^n \Big)_H \\
      &= 4 \Big( \frac{3}{2} X^n- 2X^{n-1} + \frac{1}{2} X^{n-2},
      \Theta^n \Big)_H\\
      &\quad - 4 \Big( \frac{3}{2} X(t_n) - 2 X(t_{n-1}) + \frac{1}{2}
      X(t_{n-2}), \Theta^n \Big)_H\\
      &\quad + 2 \big( 3 \Xi^n - 4 \Xi^{n-1} + \Xi^{n-2}, \Xi^n \big)_H.
    \end{split}
  \end{align}
  Using the definitions of the numerical scheme \eqref{BDF2 scheme} and of
  the exact solution \eqref{sde solution} to equation \eqref{sde_problem}, we
  deduce further  
  \begin{align*}
    \Gamma^n &= -4k \big\langle A(X^n), \Theta^n \big\rangle_{V^* \times V}\\
    &\qquad + 2 \big( 3  B(X^{n-1})\Delta_k W^n
    - B(X^{n-2})\Delta_k W^{n-1}, \Theta^n \big)_H\\ 
    &\qquad + 2 \Big\langle 3 \int_{t_{n-1}}^{t_n} A(X(s)) \diff s
    - \int_{t_{n-2}}^{t_{n-1}} A(X(s)) \diff s,
    \Theta^n \Big\rangle_{V^* \times V}\\ 
    &\qquad - 2 \Big( 3 \int_{t_{n-1}}^{t_n} B(X(s)) \diff W(s)
    - \int_{t_{n-2}}^{t_{n-1}} B(X(s)) \diff W(s),
    \Theta^n \Big)_H\\
    &\qquad + 2 \big( 3 \Xi^n - 4 \Xi^{n-1} + \Xi^{n-2}, \Xi^n \big)_H,
  \end{align*}
  where we do not distinguish notationally between the 
  solution $X$ and its modification appearing in \eqref{sde solution}, since
  we eventually take expectations of these terms. 

  After rearranging the terms, we arrive at $\Gamma^n = \Gamma_1^n
  + \Gamma_2^n + \Gamma_3^n + \Gamma_4^n + \Gamma_5^n$ with 
  \begin{align*}
    \Gamma_1^n &\coloneqq -4k \langle A(X^n)-A(X(t_n)),
    \Theta^n \rangle_{V^* \times V},\\ 
    \Gamma_2^n &\coloneqq 2 \big( 3\big[B(X^{n-1})-B(X(t_{n-1}))\big]
    \Delta_k W^n, \Theta^n \big)_H\\ 
    &\qquad \qquad -2 \big( \big[B(X^{n-2})-B(X(t_{n-2}))\big]
    \Delta_k W^{n-1}, \Theta^n \big)_H,\\ 
    \Gamma_3^n &\coloneqq 2 \Big\langle 3 \int_{t_{n-1}}^{t_n}
    A(X(s)) \diff s - \int_{t_{n-2}}^{t_{n-1}} A(X(s)) \diff s
    - 2k A(X(t_n)),\Theta^n \Big\rangle_{V^* \times V},\\ 
    \Gamma_4^n &\coloneqq 2 \Big( 3 \int_{t_{n-1}}^{t_n}
    B(X(t_{n-1}))-B(X(s)) \diff W(s), \Theta^n \Big)_H\\
    &\qquad \qquad - 2 \Big( \int_{t_{n-2}}^{t_{n-1}}
    B(X(t_{n-2}))-B(X(s)) \diff W(s), \Theta^n \Big)_H,\\
    \Gamma_5^n &\coloneqq 2 \big( 3 \Xi^n - 4 \Xi^{n-1} + \Xi^{n-2}, \Xi^n
    \big)_H. 
  \end{align*}
  We will further estimate each $\Gamma_i^n$ for $i \in \{1,\ldots,5\}$
  separately. 
  
  The assumption \eqref{con:str mono AB} is essential to estimate $\Gamma_1^n$
  appropriately. Together with the Lipschitz continuity of the operator $A$ and
  an application of Young's inequality, we conclude that
  \begin{align*}
    \Gamma_1^n &= -4k \langle A(X^n)-A(X(t_n)), E^n \rangle_{V^* \times V}
    +4k \langle  A(X^n)-A(X(t_n)), \Xi^n \rangle_{V^* \times V}\\ 
    &\leq -2kK \|E^n\|^2_V -2k\nu\|B(X^n)-B(X(t_n))\|^2_{\LB_2(U_0,H)}
    + 2k\kappa \|E^n\|_H^2\\
    &\quad + 4kL \|E^n\|_V \|\Xi^n\|_V\\
    &\leq -\frac{3}{2} kK \|E^n\|^2_V -2k\nu\|\Delta B^n\|^2_{\LB_2(U_0,H)}
    + 2k\kappa \|E^n\|_H^2
    + 8k\frac{L^2}{K} \|\Xi^n\|^2_V,
  \end{align*}
  where we also made use of the notation $\Delta B^n \coloneqq
  B(X^n)-B(X(t_n))$ for $n \in \{0,\ldots,N_k\}$.
  Regarding $\Gamma_2^n$, some elementary calculations yield that
  \begin{align*}
    \Gamma_2^n &= 2 \big( 3\Delta B^{n-1}\Delta_k W^n-\Delta B^{n-2}\Delta_k
    W^{n-1}, \Theta^n \big)_H\\ 
    &= 2 \big( \Delta B^{n-1}\Delta_k W^n-\Delta B^{n-2}\Delta_k W^{n-1},
    \Theta^n-2\Theta^{n-1}+\Theta^{n-2} \big)_H\\ 
    &\quad + 2 \big( \Delta B^{n-1}\Delta_k W^n, 2\Theta^n-\Theta^{n-1} \big)_H
    -2 \big( \Delta B^{n-2}\Delta_k W^{n-1}, 2\Theta^{n-1}-\Theta^{n-2} 
    \big)_H\\ 
    &\quad + 2 \big( \Delta B^{n-1}\Delta_k W^n, 3\Theta^{n-1}-\Theta^{n-2}
    \big)_H. 
  \end{align*}
  By using the martingale property of the stochastic integral, it follows that
  \begin{equation*}
    \E \big[ \Delta B^{n-1}\Delta_k W^n \mid \F_{t_{n-1}} \big]
    = \E \Big[ \int_{t_{n-1}}^{t_n} B(X^{n-1})-B(X(t_{n-1})) \diff W(s) 
    \mid \F_{t_{n-1}} \Big] = 0.
  \end{equation*}
  This fact and the $\F_{t_{n-1}}$-measurability of the random variable
  $(3\Theta^{n-1}-\Theta^{n-2})$ imply
  \begin{equation*}
    \E \big[ \big( \Delta B^{n-1}\Delta_k W^n, 3\Theta^{n-1}-\Theta^{n-2}
    \big)_H \big] = 0.
  \end{equation*}
  By applying Young's inequality to the decomposition of $\Gamma_2$ and
  taking expectation, we conclude that
  \begin{align*}
    \E\big[\Gamma_2^n\big]
    &\leq \nu \E \big[\|\Delta B^{n-1}\Delta_k W^n-\Delta B^{n-2}
    \Delta_k W^{n-1}\|_H^2\big]\\ 
    &\quad + \frac{1}{\nu} \E \big[\|\Theta^n-2\Theta^{n-1}
    +\Theta^{n-2}\|_H^2\big] + 2 \E \big[\big( \Delta B^{n-1}\Delta_k W^n,
    2\Theta^n-\Theta^{n-1} \big)_H\big]\\ 
    &\quad - 2 \E \big[\big( \Delta B^{n-2}\Delta_k W^{n-1},
    2\Theta^{n-1}-\Theta^{n-2} \big)_H\big].
  \end{align*}
  Next, observe that 
  \begin{align*}
    &\E \big[\|\Delta B^{n-1}\Delta_k W^n- \Delta B^{n-2}
    \Delta_k W^{n-1}\|_H^2\big]\\
    &\quad = \E \big[\|\Delta B^{n-1}\Delta_k W^n \|^2_H \big]
    + \E \big[ \| \Delta B^{n-2} \Delta_k W^{n-1}\|_H^2\big]
  \end{align*}
  which follows again from the martingale property of the stochastic integral.
  Hence, after summing over $n$ from $2$ to $j \in \{2,\ldots, N_k\}$ we arrive
  at
  \begin{align*}
    \sum_{n = 2}^j \E\big[\Gamma_2^n\big]
    &\leq \nu \sum_{n = 2}^j \big( \E \big[\|\Delta B^{n-1}\Delta_k W^n \|^2_H
    \big] + \E \big[ \| \Delta B^{n-2} \Delta_k W^{n-1}\|_H^2\big] \big)\\
    &\quad + 2 \E \big[\big( \Delta B^{j-1}\Delta_k W^j,
    2\Theta^j-\Theta^{j-1} \big)_H - \big( \Delta B^{0}\Delta_k W^{1},
    2\Theta^{1}-\Theta^{0} \big)_H\big] \\
    &\quad + \frac{1}{\nu} \sum_{n = 2}^j
    \E \big[\|\Theta^n-2\Theta^{n-1} +\Theta^{n-2}\|_H^2\big].
  \end{align*}
  From applications of the Cauchy--Schwarz inequality, Young's inequality and
  the It\=o isometry we obtain
  \begin{align*}
    &2\E \big[\big( \Delta B^{j-1}\Delta_k W^j,
    2\Theta^j-\Theta^{j-1} \big)_H \big]\\ 
    &\quad \le \frac{1}{\nu} \E\big[\|2\Theta^j-\Theta^{j-1}\|^2_H\big] 
    + \nu \E\big[ \| \Delta B^{j-1} \Delta_k W^{j-1}\|_H^2\big]\\ 
    &\quad = \frac{1}{\nu} \E\big[\|2\Theta^j-\Theta^{j-1}\|^2_H\big] 
    + k\nu \E\big[ \| \Delta B^{j-1} \|_{\mathcal{L}_2(U_0,H)}^2\big].
  \end{align*}
  The term $\E\big[\big( \Delta B^{0}\Delta_k W^{1},
  2\Theta^{1}-\Theta^{0}\big)_H\big]$ can be estimated in the same way.
  Inserting this into the estimate of $\Gamma_2^n$, applying again the It\=o
  isometry and recalling that $\nu \in (1,\infty)$ then finally yields the
  estimate  
  \begin{align*}
    \sum_{n = 2}^j \E\big[\Gamma_2^n\big]
    &\leq \frac{1}{\nu} \Big( \E\big[\|2\Theta^j-\Theta^{j-1}\|^2_H\big]
    + \sum_{n = 2}^j \E \big[\|\Theta^n-2\Theta^{n-1} +\Theta^{n-2}\|_H^2\big]
    \Big)\\
    &\quad + 2 k \nu \sum_{n = 1}^j \E \big[ \| \Delta B^{n-1}
    \|_{\mathcal{L}_2(U_0,H)}^2 \big]
    + \E\big[\|2\Theta^1-\Theta^0\|^2_H\big].
  \end{align*}

  Next, we turn to the estimation of $\Gamma_3^n$. We apply Young's inequality
  and the Lipschitz continuity of $A$ to deduce that
  \begin{align*}
    \Gamma_3^n &= 4 \int_{t_{n-1}}^{t_n} \langle A(X(s))-A(X(t_n)), \Theta^n
    \rangle_{V^* \times V} \diff s\\
    &\quad + 2 \int_{t_{n-1}}^{t_n} \langle A(X(s))-A(X(t_{n-1})), \Theta^n
    \rangle_{V^* \times V}  \diff s\\ 
    &\quad - 2 \int_{t_{n-2}}^{t_{n-1}} \langle A(X(s))-A(X(t_{n-1})),
    \Theta^n \rangle_{V^* \times V}  \diff s\\ 
    &\leq \frac{16}{K} \int_{t_{n-1}}^{t_n} \|A(X(s))-A(X(t_n))\|_{V^*}^2
    \diff s\\ 
    &\quad + \frac{8}{K} \int_{t_{n-2}}^{t_n}
    \|A(X(s))-A(X(t_{n-1}))\|_{V^*}^2 \diff s
    + \frac{1}{2} kK \|\Theta^n\|_V^2.
  \end{align*}
  After taking expectations and applying Fubini's theorem, we arrive at
  \begin{align*}
    \E\big[ \Gamma_3^n \big]
    &\leq \frac{16}{K} \int_{t_{n-1}}^{t_n}
    \E \Big[ \|A(X(s))-A(X(t_n))\|_V^2 \Big] \diff s\\
    &\quad + \frac{8}{K} \int_{t_{n-2}}^{t_n}
    \E \Big[ \|A(X(s))-A(X(t_{n-1}))\|_V^2 \Big] \diff s\\
    &\quad + kK \big( \E\big[ \|E^n\|_V^2 \big]
    + \E\big[ \|\Xi^n\|_V^2 \big] \big).
  \end{align*}  
  Therefore, the summation over $n$ from $2$ to
  $j \in \{2,\ldots,N_k\}$ together with an application of
  Lemma~\ref{lem:est_int_A} (and an obvious modification thereof) shows that
  \begin{equation*}
    \sum_{n=2}^{j} \E\big[ \Gamma_3^n \big]
    \le kK \sum_{n=2}^{j} \Big( \E\big[ \|E^n\|_V^2 \big]
    + \E\big[ \|\Xi^n\|_V^2 \big] \Big)
    + 32 \frac{L_A}{K} k^{\frac{2}{\p}}
    \|X\|^2_{\p-\mathrm{var},L^2(\Omega;V)}.
  \end{equation*}
  
  To decompose the term $\Gamma_4^n$, we define $\mathrm{I}^n \coloneqq
  \int_{t_{n-1}}^{t_n} B(X(t_{n-1}))-B(X(s)) \diff W(s)$ for $n \in
  \{1,\ldots,N_k\}$. An elementary calculation shows that
  \begin{align*}
     \Gamma_4^n &= 2 \big( 3 \mathrm{I}^n-\mathrm{I}^{n-1}, \Theta^n \big)_H\\
     &= 2 \big( \mathrm{I}^n-\mathrm{I}^{n-1},
     \Theta^n-2\Theta^{n-1}+\Theta^{n-2} \big)_H\\
     &\quad +2 \big( \mathrm{I}^n, 2\Theta^n-\Theta^{n-1} \big)_H
     -2 \big( \mathrm{I}^{n-1}, 2\Theta^{n-1}-\Theta^{n-2} \big)_H\\
     &\quad +2 \big( \mathrm{I}^n, 3\Theta^{n-1}-\Theta^{n-2} \big)_H.
  \end{align*}
  In the same way as in the estimation of $\Gamma_2^n$, we get $\E \big[\big(
  \mathrm{I}^n, 3\Theta^{n-1}-\Theta^{n-2} \big)_H\big] = 0$. After applying
  Young's inequality and taking expectation, we arrive at
  \begin{align*}
     \E\big[\Gamma_4^n\big] &\leq \frac{\nu}{\nu-1}
     \E\big[\|\mathrm{I}^n-\mathrm{I}^{n-1}\|_H^2\big] + \frac{\nu-1}{\nu}
     \E\big[\|\Theta^n-2\Theta^{n-1}+\Theta^{n-2}\|_H^2\big]\\
     &\quad +2 \E\big[\big( \mathrm{I}^n, 2\Theta^n-\Theta^{n-1} \big)_H\big]
     -2 \E\big[\big( \mathrm{I}^{n-1},
     2\Theta^{n-1}-\Theta^{n-2} \big)_H\big].
  \end{align*}
  Hence, after summing over $n$ from $2$ to $j \in \{2,\ldots,N_k\}$ we obtain
  \begin{align*}
    \sum_{n=2}^{j} \E\big[ \Gamma_4^n \big]
    &\le \frac{\nu}{\nu-1} \sum_{n = 2}^j 
    \E\big[ \|\mathrm{I}^n - \mathrm{I}^{n-1}\|_H^2\big]
    + \frac{\nu-1}{\nu} \sum_{n = 2}^j
    \E\big[\|\Theta^n-2\Theta^{n-1}+\Theta^{n-2}\|_H^2\big]\\
    &\quad +2 \E\big[\big( \mathrm{I}^j, 2\Theta^j-\Theta^{j-1} \big)_H\big]
    -2 \E\big[\big( \mathrm{I}^1, 2\Theta^1-\Theta^0 \big)_H\big]\\
    &\le \frac{\nu}{\nu-1} \sum_{n = 2}^j 
    \E\big[ \|\mathrm{I}^n - \mathrm{I}^{n-1}\|_H^2\big]
    + \frac{\nu-1}{\nu} \sum_{n = 2}^j
    \E\big[\|\Theta^n-2\Theta^{n-1}+\Theta^{n-2}\|_H^2\big]\\
    &\quad + \frac{\nu}{\nu-1} \E\big[\|\mathrm{I}^j\|_H^2\big]
    + \frac{\nu-1}{\nu} \E\big[\|2\Theta^j-\Theta^{j-1}\|_H^2\big]\\
    &\quad + \E\big[\|\mathrm{I}^1\|_H^2\big]
    + \E\big[\|2\Theta^1-\Theta^0\|_H^2\big]
  \end{align*}
  by a further application of Young's inequality. Since $I^n$ and $I^{n-1}$
  are uncorrelated and, hence, orthogonal with respect to 
  the inner product in $L^2(\Omega;H)$, it follows
  \begin{align*}
    \E \big[ \|\mathrm{I}^n - \mathrm{I}^{n-1}\|_H^2\big]
    = \E\big[\|\mathrm{I}^n \|^2_H \big]
    + \E \big[ \|\mathrm{I}^{n-1}\|_H^2\big] 
  \end{align*}
  for all $n \in \{2,\ldots,N_k\}$.
  Together with Lemma~\ref{lem:est_stoch_int} we therefore get
  \begin{align*}
    \begin{split}
      & \frac{\nu}{\nu-1} \E\big[\|\mathrm{I}^j\|_H^2\big]
      + \E\big[\|\mathrm{I}^1\|_H^2\big] + \frac{\nu}{\nu-1}
      \sum_{n=2}^{j} \E \big[\| \mathrm{I}^n -\mathrm{I}^{n-1} \|_H^2\big]\\
      &\quad \leq \frac{2\nu}{\nu-1}
      \sum_{n=1}^{j} \E\big[\|\mathrm{I}^n\|_H^2\big] 
      \leq \frac{2\nu}{\nu-1}
      L_B k^{\frac{2}{\p}} \|X\|^2_{\p-\mathrm{var},L^2(\Omega;V)}.
    \end{split}
  \end{align*}
  Altogether, this gives the estimate
  \begin{align*}
    \sum_{n=2}^{j} \E\big[ \Gamma_4^n \big]
    &\le \frac{\nu-1}{\nu} \Big( \E\big[\|2\Theta^j-\Theta^{j-1}\|_H^2\big]
    + \sum_{n=2}^{j} \E\big[\|\Theta^n-2\Theta^{n-1}+\Theta^{n-2}\|_H^2\big]
    \Big)\\
    &\quad + \frac{2\nu}{\nu-1}
    L_B k^{\frac{2}{\p}} \|X\|^2_{\p-\mathrm{var},L^2(\Omega;V)}
    + \E\big[\|2\Theta^1-\Theta^0\|_H^2\big]
  \end{align*}
  for every $j \in \{2,\ldots,N_k\}$.

  Finally, the term $\Gamma_5^n$ is rewritten in terms of the identity
  \eqref{eqn:id 2-step} by
  \begin{align*}
    \Gamma_5^n &=  \|\Xi^n\|^2_H - \|\Xi^{n-1}\|^2_H +
    \|2 \Xi^n- \Xi^{n-1}\|^2_H - \|2 \Xi^{n-1}- \Xi^{n-2}\|^2_H\\ 
    &\quad + \|\Xi^n- 2 \Xi^{n-1} + \Xi^{n-2}\|^2_H. 
  \end{align*}
  After taking expectation and summing over $n$ from $2$ to
  $j \in \{2,\ldots,N_k\}$ in equation \eqref{eq:Gamma_n}, we also see that
  \begin{align*}
    &\E\big[\| E^j \|^2_H\big] - \E\big[\| E^1\|^2_H\big]
    + \E\big[\|2 E^j- E^{j-1}\|^2_H\big]
    - \E\big[\|2E^1-E^0\|^2_H\big]\\
    &\quad \leq \sum_{n=2}^{j} \Big( \E\big[\Gamma_1^n + \Gamma_2^n
    + \Gamma_3^n + \Gamma_4^n + \Gamma_5^n \big]
    - \E \big[\| E^n-2E^{n-1}+ E^{n-2}\|^2_H\big] \Big). 
  \end{align*}
  Now, we insert the estimates for $\Gamma_i^n$, $i=1,\ldots,5$, and 
  we use that
  \begin{align*}
    &\E \big[\| E^n-2E^{n-1}+ E^{n-2}\|^2_H\big]\\
    &\quad = \E \big[\|\Theta^n-2\Theta^{n-1}+\Theta^{n-2}\|_H^2\big]
    + \E \big[ \|\Xi^n- 2 \Xi^{n-1} + \Xi^{n-2}\|^2_H\big],
  \end{align*}
  which follows from \eqref{eq:E_n H-norm}. This gives for every $j \in
  \{2,\ldots,N_k\}$ that
  \begin{align*}
    &\E\big[\| E^j\|^2_H\big] - \E\big[\| E^1 \|^2_H\big]
    + \E\big[\|2 E^j-E^{j-1}\|^2_H\big] - \E\big[\|2 E^1-E^0 \|^2_H\big]\\
    &\quad \leq -\frac{1}{2} k K \sum_{n=2}^{j} \E\big[\|E^n\|^2_V\big]
    + 2k \kappa \sum_{n=2}^{j} \E\big[\|E^n\|^2_H\big]
    + k \Big( 8\frac{L^2}{K}+K \Big) 
    \sum_{n=2}^{j} \E\big[\|\Xi^n\|^2_V\big]\\
    &\qquad + \E\big[\|2\Theta^j-\Theta^{j-1}\|^2_H\big]
    + 2 \E\big[\|2\Theta^1-\Theta^0\|^2_H\big] + 2k\nu \sum_{n=0}^{1}
    \E \big[\|\Delta B^n\|^2_{\LB_2(U_0,H)}\big] \\
    &\qquad + \E\big[ \|\Xi^j\|^2_H \big] - \E\big[ \|\Xi^{1}\|^2_H \big]
    + \E \big[ \|2\Xi^j-\Xi^{j-1}\|^2_H \big]
    - \E \big[\|2\Xi^{1}-\Xi^{0}\|^2_H \big]\\
    &\qquad + \Big(32 \frac{L_A}{K} + \frac{2\nu L_B}{\nu-1}\Big) 
    k^{\frac{2}{\p}} \|X\|^2_{\p-\mathrm{var},L^2(\Omega;V)}. 
  \end{align*}
  Thus, after recalling \eqref{eq:E_n H-norm} and some rearranging we arrive at
  \begin{align}
    \label{est_Theta2}
    \begin{split}
      &\E\big[\| E^j\|^2_H\big]
      + \frac{1}{2} k K \sum_{n=2}^{j} \E\big[\|E^n\|^2_V\big] \\
      &\quad \leq 2k \kappa \sum_{n=2}^{j} \E\big[\|E^n\|^2_H\big]
      + \E \big[ \|\Xi^j\|^2_H \big] 
      + k \Big( 8\frac{L^2}{K}+K \Big) 
      \sum_{n=2}^{j} \E\big[\|\Xi^n\|^2_V\big]\\
      &\qquad + 3 \E\big[\|2\Theta^1-\Theta^0\|^2_H\big]
      + \E\big[\|\Theta^1\|_H^2\big]+ 2k\nu \sum_{n=0}^{1}
      \E \big[\|\Delta B^n\|^2_{\LB_2(U_0,H)}\big] \\
      &\qquad + \Big(32 \frac{L_A}{K} + \frac{2\nu L_B}{\nu-1}\Big) 
      k^{\frac{2}{\p}} \|X\|^2_{\p-\mathrm{var},L^2(\Omega;V)}.
    \end{split}
  \end{align}
  Our goal is to provide estimates for the terms involving $\Xi^n$ on the
  right-hand side. The best approximation properties of the orthogonal
  projection $P_h$ with respect to the norm in $H$ implies
  \begin{equation*}
  	\|\Xi^n\|_H = \mathrm{dist}_H(X(t_n),V_h)
  	\qquad \text{ for each }n \in \{0,\ldots,N_k\}.
  \end{equation*}
  Further, the properties of the projections $P_h$ and $R_h$ yield the
  estimate
  \begin{align*}
    \|\Xi^n\|_V &= \|(P_h-\id)X(t_n)\|_V\\
    &\leq \|P_h(\id-R_h)X(t_n)\|_V + \|(P_h R_h-\id)X(t_n)\|_V\\ 
    &\leq \big(\|P_h\|_{\LB(V)}+1\big) \|(R_h-\id)X(t_n)\|_V
    = \big(1+\|P_h\|_{\LB(V)}\big) \mathrm{dist}_V(X(t_n),V_h) 
  \end{align*}
  for each $n \in \{0,\ldots,N_k\}$. Altogether, we conclude for
  $j \in \{2,\ldots,N_k\}$ that 
  \begin{align}
  \label{ineq4:Xi^n}
    \begin{split}
      &\E\big[\|\Xi^j\|_H^2\big] + k \Big( 8\frac{L^2}{K}+K \Big)
      \sum_{n=2}^j \E\big[\|\Xi^n\|_V^2\big]\\
      &\quad \leq \E\big[\mathrm{dist}_H(X(t_j),V_h)^2\big]\\
      &\qquad + k \Big( 8\frac{L^2}{K}+K \Big)
      \big(1+\|P_h\|_{\LB(V)}\big)^2
      \sum_{n=2}^j \E\big[\mathrm{dist}_V(X(t_n),V_h)^2\big]. 
    \end{split}
  \end{align}
  Moreover, we want to estimate further the remaining terms in
  \eqref{est_Theta2} that depend on the time steps $t_n$ with
  $n \in \{0,1\}$. It holds that
  $\|2\Theta^1-\Theta^0\|^2_H \leq 5 (\|\Theta^1\|_H^2+\|\Theta^0\|^2_H)$.
  Then, it follows from the $H$-orthogonality of the decomposition
  $E^n = \Theta^n + \Xi^n$ that
  \begin{equation*}
    3 \E\big[\|2\Theta^1-\Theta^0\|^2_H\big] + \E\big[\|\Theta^1\|_H^2\big]
    \leq 16 \sum_{n=0}^{1} \E\big[\|E^n\|_H^2\big].
  \end{equation*}
  This estimate together with Assumption~\ref{assump IC2} yields
  \begin{align*}
    \begin{split}
      &3 \E\big[\|2\Theta^1-\Theta^0\|^2_H\big] + \E\big[\|\Theta^1\|_H^2\big]
      + 2k\nu \sum_{n=0}^{1} \E \big[\|\Delta B^n\|^2_{\LB_2(U_0,H)}\big]\\
      &\quad \leq C_I \max\big\{16,2\nu\big\}
      (k + h^\gamma).
    \end{split}
  \end{align*}
  By inserting the last inequality and \eqref{ineq4:Xi^n} into estimate
  \eqref{est_Theta2}, we arrive at
  \begin{align*}
    &(1-2k\kappa) \E\big[\|E^j\|^2_H\big] + kK \sum_{n = 2}^j
    \E \big[\|E^n\|_V^2 \big]\\
    &\quad\le 2k\kappa \sum_{n=2}^{j-1} \E\big[\|E^n\|^2_H\big]
    + \tilde{C} \, \E\big[\mathrm{dist}_H(X(t_j),V_h)^2\big]\\
    &\qquad + \tilde{C} k \big(1+\|P_h\|_{\LB(V)}\big)^2
    \sum_{n=2}^j \E\big[\mathrm{dist}_V(X(t_n),V_h)^2\big]\\
    &\qquad 
    + \tilde{C} (k^{\frac{2}{\p}} \|X\|^2_{\p-\mathrm{var},L^2(\Omega;V)}
    + k + h^\gamma \big),
  \end{align*}
  where the constant $\tilde{C}>0$ is defined by
  \begin{equation*}
    \tilde{C} = \max\Big\{ 1, 8\frac{L^2}{K}+K, 16 C_I, 2\nu C_I,
    32\frac{L_A}{K} + \frac{2\nu L_B}{\nu-1} \Big\}.
  \end{equation*}
  Finally, applying a discrete version of Gronwall's inequality, see, e.g.,
  \cite{clark1987}, shows
  \begin{align*}
    &\E\big[\|E^j\|^2_H\big] + k \sum_{n = 2}^j \E \big[ \|E^n\|_V^2 \big]\\
    &\quad \le \ee^{2\kappa t_j \tilde{C}_k^{-1}} \frac{\tilde{C}}{\tilde{C}_k}
    \Big(k + h^\gamma+ k^{\frac{2}{\p}} \|X\|^2_{\p-\mathrm{var},L^2(\Omega;V)}
    + \E\big[\mathrm{dist}_H(X(t_j),V_h)^2\big] \\
    &\qquad + k \big(1+\|P_h\|_{\LB(V)}\big)^2 \sum_{n=2}^j
    \E\big[\mathrm{dist}_V(X(t_n),V_h)^2\big]
     \Big),
  \end{align*}
  where $\tilde{C}_k \coloneqq \min\{1-2k\kappa,K\}$. Taking the maximum with
  respect to $j \in \{2,\ldots,N_k\}$ on the right-hand side of this inequality
  yields an estimate for each summand on the left-hand side and completes the
  proof.
\end{proof}

\begin{remark}
  \label{rmk:proj norm}
  The error estimate in Theorem~\ref{thm:err} depends on the term
  $\|P_h\|_{\LB(V)}$ which is, in general, not uniformly bounded for arbitrarily
  small $h \in (0,1)$. However, for many important examples of evolution
  equations and Galerkin schemes a uniform bound can indeed be given. For
  instance, for the finite element method and the typical choice of the Gelfand
  triple with $V = H^1_0(\mathcal{D})$ and $H = L^2(\mathcal{D})$ on a bounded
  domain $\mathcal{D} \subset \R^d$, $d \in \{1,2,3\}$,
  the $H^1$-stability of the orthogonal $L^2$-projection $P_h$ has been
  investigated in \cite{bank2014,bramble2002,carstensen2002,crouzeix1987}.
\end{remark}

\begin{remark}
  \label{rmk:err est BEM}
  Under Assumption~\ref{assump AB2} and Assumption~\ref{assump X}, the BEM
  scheme \eqref{BEM scheme} with initial value $X^0_{k,h} \in L^2(\Omega;V)$
  admits for sufficiently small temporal step size with $k\kappa<1$ an
  approximate solution $(X^n_{k,h})_{n=0}^{N_k}$ such that at time $t_1 = k$ 
  the error estimate
  \begin{align*}  
    &\|X^1_{k,h}-X(t_1)\|_{L^2(\Omega;H)}^2
    + k \|X^1_{k,h}-X(t_1)\|_{L^2(\Omega;V)}^2\\
    &\quad \leq \frac{C}{C_k}
    \Big( k + \|(P_h-\id)X(t_1)\|_{L^2(\Omega;H)}^2
    + k \|(P_h-\id)X(t_1)\|_{L^2(\Omega;V)}^2\\
    &\qquad + \|X^0_{k,h}-X(t_0)\|_{L^2(\Omega;H)}^2
    + k \|B(X^0_{k,h})-B(X(t_0))\|_{L^2(\Omega;\LB_2(U_0,H))}^2 \Big)
  \end{align*}
  holds, where $C_k \coloneqq \min\{1-k\kappa,K\}$ and $C > 0$ is a constant
  only depending on $\kappa,\nu,L,K,T$ and $\beta_{V \hookrightarrow H}$. This
  estimate can be proven with similar techniques as used in the proof of
  Theorem~\ref{thm:err}.
\end{remark}

\section{Numerical experiments}
\label{sec:simu}

In this section, we perform two numerical experiments to give a more practical 
assessment of the BDF2-Maruyama scheme \eqref{BDF2 scheme}.
In Subsection~\ref{subsec:SHE} we use the scheme to simulate
the stochastic heat equation with additive noise and in
Subsection~\ref{subsec:SNP} we consider a stochastic partial differential
equation with a quasilinear drift as well as nonlinear multiplicative noise.
To better illustrate its performance, we compare the BDF2-Maruyama scheme to
the BEM scheme \eqref{BEM scheme}. 

In all numerical experiments, we use equidistant grids to discretize the
time-space domain $[0,T] \times [0,1]$. Regarding the temporal 
discretization, the BEM scheme \eqref{BEM scheme} as well as the BDF2-Maruyama
scheme \eqref{BDF2 scheme} are applied with the equidistant temporal step size
$k = \frac{T}{N_k}$, where $N_k = 2^l$ for $l=5,\ldots,10$. The spatial
discretization is realized by using the standard finite element method. To be
more precise, we consider the equidistant partition 
$\{x_i = ih \mid i = 0,\ldots,N_h+1\}$ with $N_h = 2^{12}$ interior nodes
and spatial step size $h = \frac{1}{N_h+1}$. We define the space $V_h$
consisting of piecewise linear finite elements by
\begin{equation*}
  V_h = \big\{v \in C([0,1]) \colon v|_{[x_{i-1},x_i]} \in \mathcal{P}_1
  \, \forall i=1,\ldots,N_h+1, \, v(0)=v(1)=0 \big\},
\end{equation*}
where $\mathcal{P}_1$ denotes the set of all polynomials up to degree $1$.
By $\{\phi_i\}_{i=1}^{N_h} \subset V_h$ we denote the Lagrange basis functions
of $V_h$ which are uniquely determined by $\phi_i(x_j)=\delta_{ij}$ for all 
$i,j = 1,\ldots, N_h$. Further, we recall from \cite[Section~4.4]{brenner2008}
or \cite[Section~5.1]{larsson2003} that the family of spaces
$\{V_h\}_{h \in (0,1)}$ defines a Galerkin scheme for the Sobolev space
$V = H^1_0(0,1)$. From \cite[Theorem~2]{crouzeix1987} it follows that
$\|P_h\|_{\LB(V)} < \infty$ holds uniformly in $h \in (0,1)$.

Moreover, we denote by 
\begin{equation}
  \label{eq5:intop}
  I_h \colon C([0,1]) \to V_h, 
  \quad v \mapsto I_h(v) = \sum_{i=1}^{N_h} v(x_i) \phi_i.
\end{equation}
the interpolation operator. An explicit calculation verifies the
interpolation error estimate 
\begin{equation}
  \label{ineq:I_h err L2}
  \|I_h(v) - v\|_{L^2(0,1)} \leq Ch \|v\|_{H_0^1(0,1)}
\end{equation}
for every function $v \in H_0^1(0,1) \hookrightarrow C([0,1])$ and
some constant $C > 0$, see, e.g., \cite[Theorem~4.4.20]{brenner2008}. 

Regarding the simulation of the $U$-valued $Q$-Wiener process $W$, we follow
\cite[Section~10.2]{lord2014} and consider the Karhu\-nen--Lo\`eve expansion
\eqref{eq2:KHexp}.
Notice that the decay of the eigenvalues $(q_j)_{j \in \N}$ determines the
smoothness of the Wiener process regarding the spatial variable. In the case of
$U = L^2(0,1)$, the choice of the sine basis $\chi_j(x) = \sqrt{2} \sin(j\pi
x)$ and the eigenvalues $q_j = j^{-(2r+1+\varepsilon)}$ with $\varepsilon > 0$
and $r \in \R_+$ leads to an almost surely $H^r_0(0,1)$-valued Wiener process,
see \cite[Example~10.9]{lord2014}. This enables us to sample efficiently the
Wiener process $W$, since the corresponding truncated representation
\begin{equation}
  \label{WP_simu}
  W^J(t,x) = \sum_{j=1}^J \sqrt{2} j^{-\frac{1}{2}(2r+1+\epsilon)} \beta_j(t)
  \sin(j\pi x), \quad J \in \N,
\end{equation}
can be implemented by using a discrete sine transform. The truncation
parameter is chosen to be $J = 2^{12}$ in all simulations.

In our numerical experiments, we compute the strong error between the
approximate solution of the respective 
scheme and the exact solution of the stochastic evolution equation
\eqref{sde_problem} with respect to the $L^\infty([0,T];L^2(\Omega;H))$-norm.
Hereby, we only take the maximum over the points of the temporal grid of the
considered numerical approximation. Using also a Monte Carlo simulation with
$M = 10^4$ independent samples, we approximate the strong error by 
\begin{align}
  \label{eq5:str_err}
  \begin{split}
    \mathrm{error}_{k,h} &= \max_{n\in\{2,\ldots,N_k\}} \Big( \frac{1}{M}
    \sum_{m=1}^M \|X^{n,(m)}_{k,h} - X^{(m)}(nk)\|_H^2 \Big)^{\frac{1}{2}}\\
    &\approx \max_{n\in\{2,\ldots,N_k\}} \|X^n_{k,h}-X(t_n)\|_{L^2(\Omega;H)},
  \end{split}
\end{align}
where $\{X^{n,(m)}_{k,h} - X^{(m)}(nk)\}_{m = 1,\ldots,M}$ are independently
generated samples of the error $X^n_{k,h} - X(t_n)$. Notice that the
computation of the strong error in \eqref{eq5:str_err} is not explicitly
depending on the initial values at the two grid points $\{t_0, t_1\}$. The
reason for this is that the same initial values are used for both considered
schemes, cf. Remark~\ref{rmk:ini val BDF2}. Therefore, in \eqref{eq5:str_err}
we only measure the error for all temporal grid points where the two schemes
differ.  

As a substitute for the exact solution in \eqref{eq5:str_err}, we use a
numerical reference solution which is computed by using the BDF2-Maruyama
scheme with $N_k = 2^{15}$ steps and the same number $N_h = 2^{12}$ of degrees
of freedom in all simulations. We mention that the numerical results reported
further below are not qualitatively impacted if the BEM scheme is used for the
computation of the reference solution. Moreover, to validate the statistical
significance of our numerical results, we determine the asymptotically valid
$(1-\alpha)$-confidence interval for $\|X^n_{k,h}-X(t_n)\|_{L^2(\Omega;H)}$
with $\alpha = 0.05$ for some value of the index $n \in \{2,\ldots,N_k\}$ at
which the error estimator $\text{error}_{k,h}$ in \eqref{eq5:str_err} attains
its maximum. In detail, we compute the confidence interval (CI) using the
formula
\begin{equation*}
  \Big[ \Big(\overline{Y}_M - z_{(1-\frac{\alpha}{2})} \frac{S_M}{\sqrt{M}}
    \Big)^\frac{1}{2},
  \Big( \overline{Y}_M + z_{(1-\frac{\alpha}{2})} \frac{S_M}{\sqrt{M}}
    \Big)^\frac{1}{2} \Big],
\end{equation*}
where $\overline{Y}_M$ as well as $S_M$ denote the sample mean and the unbiased
sample standard deviation with respect to $M$ realizations of independent
and identically distributed copies of the random variable
$Y = \|X^n_{k,h}-X(t_n)\|_H^2$ and $z_{(1-\frac{\alpha}{2})}$ denotes the
$(1-\frac{\alpha}{2})$-quantile of the standard normal distribution.

We also compute the \emph{experimental order of convergence} (EOC) as an
estimator of the temporal convergence rates. We define the EOC for 
successive temporal step sizes $k_{i-1}$, $k_i$ and fixed spatial step size $h$
by 
\begin{equation*}
  \mathrm{EOC} =
  \frac{\log(\mathrm{error}_{k_i,h})-\log(\mathrm{error}_{k_{i-1},h})} 
    {\log(k_i)-\log(k_{i-1})}.
\end{equation*}

\subsection{The Stochastic Heat Equation with additive noise}
\label{subsec:SHE}

We examine again the stochastic heat equation introduced in
Example~\ref{ex:SHE}. As above, we consider $U = L^2(0,1)$ and a $U$-valued
Wiener process which takes almost surely values in $H^r_0(0,1) \subset U$ for
some $r \in \R_+$.

Recall from Example~\ref{ex:SHE}, that the operators $A$ and $B$ satisfy
Assumption \ref{assump AB1}. Since the operator $A \colon V \to
V^\ast$ is linear and bounded, it is also Lipschitz continuous with $L=1$ and
the stronger monotonicity condition \eqref{con:str mono AB} holds with
$\kappa=0$, $\nu \in (1,\infty)$ and $K \in (0,2]$. Therefore,
Assumption~\ref{assump AB2} is satisfied. 

Moreover, it was shown in \cite[Theorem~2.31]{kruse2014} that the linear
problem \eqref{prob:SHE} admits for every $r \in (0,\infty)$ a unique mild
solution $X$ which is H\"older continuous with exponent
$\gamma = \min\{\frac12, \frac{r}{2}\}$
with respect to the $L^2(\Omega;H^1_0(0,1))$-norm. Since
the unique solution to \eqref{prob:SHE} as defined in Section~\ref{sec:discret}
coincides with the mild solution, cf.~\cite[Chapter~6]{daprato1992},
Assumption~\ref{assump X} is satisfied.

In contrast to Example~\ref{ex:SHE}, we do not project the initial value 
$X_0:=\sin(\pi \cdot) \in V = H^1_0(0,1)$ onto the subspace $V_h$ by applying
$P_h$. Instead we make use of the interpolation operator \eqref{eq5:intop},
which is easier to implement. Hence, we set $X^0_{k,h} := I_h(X_0) \in
L^2(\Omega;V_h)$ as the first initial value for both schemes.
From \eqref{ineq:I_h err L2} we obtain the estimate
\begin{align*}  
  \|X^0_{k,h} - X(t_0)\|_{L^2(\Omega;H)}^2
  &\le C h^2 \| X_0 \|_{L^2(\Omega;V)}^2.
\end{align*}
As discussed in Remark~\ref{rmk:ini val BDF2}, the second initial value
$X^1_{k,h}$ required for the BDF2-Maruyama scheme is computed by performing one
step with the BEM scheme. Due to Remark~\ref{rmk:err est BEM} and the previous
estimate, it holds
\begin{align*}  
  \|X^1_{k,h}-X(t_1)\|_{L^2(\Omega;H)}^2
  &\leq C \big( k + h^2 + \|(P_h-\id)X(t_1)\|_{L^2(\Omega;H)}^2\\
  &\qquad + k \, \|(P_h-\id)X(t_1)\|_{L^2(\Omega;V)}^2 \big).
\end{align*}
The best approximation property of $P_h$ in $V_h$ with respect to the norm in
$H$ and a further application of \eqref{ineq:I_h err L2} yield
\begin{equation*}
  \|(P_h-\id)X(t_1)\|_{L^2(\Omega;H)}^2
  \leq \|(I_h-\id)X(t_1)\|_{L^2(\Omega;H)}^2
  \leq C h^2 \|X(t_1)\|_{L^2(\Omega;V)}^2.
\end{equation*}
Moreover, we deduce
\begin{equation*}
  k \, \|(P_h-\id)X(t_1)\|_{L^2(\Omega;V)}^2\\
  \leq k \big(\|P_h\|_{\LB(V)}+1\big)^2 \|X(t_1)\|_{L^2(\Omega;V)}^2.
\end{equation*}
Since $X(t_1) \in L^2(\Omega;V)$ and the operator norm
$\|P_h\|_{\LB(V)}$ is uniformly bounded for $h \in (0,1)$,
Assumption~\ref{assump IC2} is fulfilled with $\gamma=2$. In particular,
both considered numerical schemes are well-defined.

Altogether, this shows that Theorem~\ref{thm:err} is applicable to problem
\eqref{prob:SHE}. Hence, assuming that the spatial step size $h$ is chosen
sufficiently small, we expect in our temporal error analysis that the strong
error of the BDF2-Maruyama scheme converges at least with rate $\frac{1}{2}$
as $k \to 0$. Furthermore, the BEM scheme is also expected to converge at
least with the same rate of $\frac12$, see \cite[Theorem~6.1]{gyoengy2009}.

Let the temporal step size $k$ and the spatial step size $h$ be given and fixed.
For our numerical experiments we have to project the Wiener process
onto the Galerkin space $V_h$, which requires the evaluation of the terms
$(\sigma \Delta_k W^{n,J},\phi_i)_H$ for each $i = 1,\ldots,N_h$.
We approximate these terms by a further application
of the interpolation operator $I_h$ 
\begin{align*}
  (\sigma \Delta_k W^{n,J},\phi_i)_H
  &\approx (I_h(\sigma \Delta_k W^{n,J}),\phi_i)_H\\
  &= \sigma \sum_{j=1}^{N_h} (\phi_i,\phi_j)_H \cdot \Delta_k W^{n,J}(x_j).
\end{align*}
Notice that $\Delta_k W^{n,J}$ takes values in $C^\infty([0,1])$
almost surely and the application of $I_h$ is well-defined.
In particular, its approximation error is sufficiently small with respect
to the $H$-norm.

Our goal is then to determine a discrete process
$(\boldsymbol{X}^n)_{n=1}^{N_k}$ consisting of random variables
$\boldsymbol{X}^n \colon \Omega \to \R^{N_h}$ such that
\begin{equation}
 \label{FEM representation}
  X^n_{k,h} = \sum_{i=1}^{N_h} \boldsymbol{X}^n_i \phi_i
\end{equation}
holds in $V_h$ for each $n=0,\ldots,N_k$.
For this, let $M_h = [(\phi_i,\phi_j)_H]_{i,j=1}^{N_h}$ and
$A_h = [(\phi_i,\phi_j)_V]_{i,j=1}^{N_h}$ denote the mass matrix and the
stiffness matrix arising from the finite element method. Then the reduced discrete
systems of \eqref{prob:SHE} for the BEM scheme and the BDF2-Maruyama scheme are
given, respectively, by
\begin{equation*}
  M_h (\boldsymbol{X}^n -\boldsymbol{X}^{n-1}) + k A_h \boldsymbol{X}^n
  = \sigma M_h \Delta_k \boldsymbol{W}^{n,J}
\end{equation*}
and
\begin{equation*}
  M_h (3\boldsymbol{X}^n - 4\boldsymbol{X}^{n-1} + \boldsymbol{X}^{n-2})
  + 2k A_h \boldsymbol{X}^n
  = \sigma M_h (3 \Delta_k \boldsymbol{W}^{n,J} 
  - \Delta_k \boldsymbol{W}^{n-1,J}),
\end{equation*}
where
$\Delta_k \boldsymbol{W}^{n,J} = [ W^J(t_n,x_i)-W^J(t_{n-1},x_i) ]_{i=1}^{N_h}$.
The discrete systems of both schemes are linear in $\boldsymbol{X}^n$ and can be
solved efficiently by using sparse matrix solvers. Due to the representation
formula \eqref{FEM representation}, the $H$-norm of the approximation $X^n_{k,h}$
can be computed by
\begin{equation*}
  \|X^n_{k,h}\|_H = \sqrt{(\boldsymbol{X}^n)^T M_h \boldsymbol{X}^n}.
\end{equation*}

We consider $T=1$ in all numerical experiments. In the first experiment, we simulate 
the deterministic heat equation \eqref{prob:SHE} with $\sigma=0$. The corresponding
results in Table~\ref{table:SHE no noise} show that the numerical error of the
BDF2-Maruyama scheme is significantly smaller compared to the error of the BEM
scheme for each level of the temporal discretization. Further, the margin
between these errors increases for larger temporal step sizes $k$ and the
BDF2-Maruyama scheme converges twice as fast as indicated by the experimental
order of convergence. Theses observations are in line with the well studied
deterministic case, see, e.g., \cite[Theorem~10.2]{thomee2006}.

\begin{table}[ht]
  \caption{Deterministic heat equation with $\sigma=0$.}
  \label{table:SHE no noise}
  \begin{tabular}{p{1cm}p{1.5cm}p{1.2cm}p{1.5cm}p{1.2cm}}
          & BEM      &       & BDF2     & \\
    \noalign{\smallskip}\hline\noalign{\smallskip}
    $N_k$ & error    & EOC   & error    & EOC \\ 
    \noalign{\smallskip}\hline\noalign{\smallskip}
    32    & 0.035361 &       & 0.020588 & \\ 
    64    & 0.018857 & 0.91  & 0.007521 & 1.45 \\ 
    128   & 0.009719 & 0.96  & 0.002289 & 1.72 \\ 
    256   & 0.004935 & 0.98  & 0.000654 & 1.81 \\ 
    512   & 0.002487 & 0.99  & 0.000176 & 1.89 \\ 
    1024  & 0.001249 & 0.99  & 0.000046 & 1.93 \\ 
    \noalign{\smallskip}\hline\noalign{\smallskip}
  \end{tabular}
\end{table}

In the following, we compare the two schemes
for fixed noise intensity $\sigma = 1$ and varying spatial regularity of the
Wiener process $W$ determined by the parameter $r \in \{0.1,1,5\}$. The
numerical results of the BEM scheme and the BDF2-Maruyama scheme are presented
for each parameter value $r$ in Table~\ref{table:SHE r=0.1} to
Table~\ref{table:SHE r=5.0}, respectively. 

In Table~\ref{table:SHE r=0.1} we see that the errors of the BDF2-Maruyama scheme
are only slightly smaller compared to those of the BEM scheme in the case of
the least regular noise with $r=0.1$. The values for the experimental order of
convergence essentially agree for both schemes. This is in line with the
expectation that a higher order temporal scheme does not provide an
advantage if the exact solution is not sufficiently regular.

In Table~\ref{table:SHE r=1.0} and Table~\ref{table:SHE r=5.0}, we notice that,
in the case of more regular noise, the BDF2-Maruyama scheme yields significantly
more accurate approximations in comparison to the BEM scheme.
The observed EOC values of both schemes exceed the expected rate of
$\frac{1}{2}$. However, this does not come as a surprise since we
discretize an evolution equation with additive noise and both schemes coincide
with their respective Milstein variants. In addition, observe that 
the BDF2-Maruyama scheme converges with a slightly higher
rate when using coarse temporal grids with $N_k \in \{32,64,128\}$. Further,
the accuracy of the BDF2-Maruyama scheme increases more clearly if 
the noise is more regular.

In conclusion, our numerical experiments indicate that the BDF2 scheme is
superior to the BEM scheme, in particular, if the noise and, hence, the exact
solution admit a certain regularity. Only in the case of less regular
noise, both schemes perform equally well. 

\begin{table}[ht]
  \caption{Stochastic heat equation with $\sigma=1$ and $r=0.1$.}
  \label{table:SHE r=0.1}
  \begin{tabular}{p{1cm}p{1.5cm}p{1.55cm}p{1.2cm}p{1.5cm}p{1.55cm}p{1.2cm}}
          & BEM & & 
          & BDF2 & & \\ 
    \noalign{\smallskip}\hline\noalign{\smallskip}
    $N_k$ & error & CI $\pm$ & EOC
          & error & CI $\pm$ & EOC \\ 
    \noalign{\smallskip}\hline\noalign{\smallskip}
    32    & 0.067292 & 0.000396 &       & 0.055539 & 0.000299 & \\ 
	64    & 0.043789 & 0.000209 & 0.62  & 0.036166 & 0.000163 & 0.62 \\ 
	128   & 0.029026 & 0.000114 & 0.59  & 0.024732 & 0.000094 & 0.55 \\ 
	256   & 0.019404 & 0.000064 & 0.58  & 0.016847 & 0.000055 & 0.55 \\ 
	512   & 0.013059 & 0.000037 & 0.57  & 0.011466 & 0.000031 & 0.56 \\ 
	1024  & 0.008803 & 0.000021 & 0.57  & 0.007773 & 0.000018 & 0.56 \\  
    \noalign{\smallskip}\hline\noalign{\smallskip}
  \end{tabular}
\end{table}

\begin{table}[ht]
  \caption{Stochastic heat equation with $\sigma=1$ and $r=1.0$.}
  \label{table:SHE r=1.0}
  \begin{tabular}{p{1cm}p{1.5cm}p{1.55cm}p{1.2cm}p{1.5cm}p{1.55cm}p{1.2cm}}
          & BEM & & 
          & BDF2 & & \\ 
    \noalign{\smallskip}\hline\noalign{\smallskip}
    $N_k$ & error & CI $\pm$ & EOC
          & error & CI $\pm$ & EOC \\ 
    \noalign{\smallskip}\hline\noalign{\smallskip}
    32    & 0.048895 & 0.000448 &       & 0.034177 & 0.000305 & \\ 
	64    & 0.026680 & 0.000226 & 0.87  & 0.016160 & 0.000128 & 1.08 \\ 
	128   & 0.014333 & 0.000111 & 0.90  & 0.008146 & 0.000055 & 0.99 \\ 
	256   & 0.007569 & 0.000055 & 0.92  & 0.004345 & 0.000026 & 0.91 \\ 
	512   & 0.003984 & 0.000027 & 0.93  & 0.002293 & 0.000012 & 0.92 \\ 
	1024  & 0.002077 & 0.000013 & 0.94  & 0.001203 & 0.000006 & 0.93 \\ 
    \noalign{\smallskip}\hline\noalign{\smallskip}
  \end{tabular}
\end{table}

\begin{table}[ht]
  \caption{Stochastic heat equation with $\sigma=1$ and $r=5.0$.}
  \label{table:SHE r=5.0}
  \begin{tabular}{p{1cm}p{1.5cm}p{1.55cm}p{1.2cm}p{1.5cm}p{1.55cm}p{1.2cm}}
	      & BEM & & 
	      & BDF2 & & \\ 
	\noalign{\smallskip}\hline\noalign{\smallskip}
	$N_k$ & error & CI $\pm$ & EOC
	      & error & CI $\pm$ & EOC \\ 
	\noalign{\smallskip}\hline\noalign{\smallskip}
	32    & 0.044139 & 0.000471 &       & 0.029223 & 0.000356 & \\ 
	64    & 0.023424 & 0.000249 & 0.91  & 0.012110 & 0.000152 & 1.27 \\ 
	128   & 0.012039 & 0.000125 & 0.96  & 0.005206 & 0.000072 & 1.22 \\ 
	256   & 0.006154 & 0.000064 & 0.97  & 0.002579 & 0.000036 & 1.01 \\ 
	512   & 0.003093 & 0.000032 & 0.99  & 0.001282 & 0.000017 & 1.01 \\ 
	1024  & 0.001563 & 0.000016 & 0.98  & 0.000640 & 0.000009 & 1.00 \\ 
	\noalign{\smallskip}\hline\noalign{\smallskip}
  \end{tabular}
\end{table}

\subsection{A Nonlinear SPDE with multiplicative noise}
\label{subsec:SNP}

In this subsection, we consider the quasilinear stochastic partial differential
equation
\begin{equation}
  \label{prob:SNP}
  \begin{aligned}
    \diff u(t,x) - &\big( \psi(|u_x(t,x)|) \cdot u_x(t,x) \big)_x \diff t\\
    &\qquad = \sigma \sqrt{8|u(t,x)|^2+1} \diff W(t,x),
    && (t,x) \in (0,T] \times (0,1),\\
    u(t,0) &= u(t,1) = 0, &&t \in (0,T],\\
    u(0,x) &= \sin(\pi x), &&x \in (0,1),
  \end{aligned}
\end{equation}
with Dirichlet boundary conditions as well as a smooth deterministic initial
value. This problem is based on a generalized example of a deterministic 
nonlinear variational problem with strongly monotone drift from
\cite[Subsection~3.5]{emmrich2004}. 

In \eqref{prob:SNP} the function $\psi\colon \R_0^+ \to \R$ is assumed to be
continuous and bounded such that the mapping
$t \mapsto t \psi(t)$ is Lipschitz continuous and strongly
monotone, i.e., there exist $m_1,m_2 > 0$ with 
\begin{equation*}
  \psi(t)t-\psi(s)s \geq m_1 (t-s) \quad \text{if } t \geq s
  \quad \text{and} \quad
  |\psi(t)t-\psi(s)s| \leq m_2 |t-s|
\end{equation*}
for all $s,t \in \R^+_0$.
Further, we assume that $m_1 \leq \psi(t) \leq m_2$ for any $t \in \R^+_0$.

Under these conditions it has been shown in \cite[Corollary~3.5.3]{emmrich2004}
that $\psi$ induces an abstract operator $A \colon V \to V^\ast$ defined by
\begin{align*}
  &A \colon V \to V^*, \quad
  v \mapsto -\big( \psi(v_x) \cdot v_x \big)_x,
\end{align*}
where we again consider the Gelfand triple with $V =
H^1_0(0,1)$ and $H = L^2(0,1)$. Moreover, the operator
$A$ is globally Lipschitz continuous and strongly monotone.

In this numerical experiment, we choose the function 
$\psi \colon \R_0^+ \to \R$ to be $\psi(t) = \mathrm{erf}(t-2)+2$, where
$\mathrm{erf}$ denotes the error function 
\begin{equation*}
  \mathrm{erf}(t) = \frac{2}{\sqrt{\pi}} \int_0^t \ee^{-s^2} \diff s,
  \quad t \in \R.
\end{equation*}

As before, the parameter $\sigma \in \R$ determines the intensity of the
nonlinear multiplicative noise in \eqref{prob:SNP}. 
The Wiener process $W$ is assumed to take values in $H^1_0(0,1)$ almost surely
and its approximation $W^J$ is defined in the same way as in \eqref{WP_simu}.
The corresponding abstract operator for the multiplicative noise is given by
\begin{align*}
  B \colon V \to \LB_2(U_0,H), \quad
    v \mapsto \sigma \sqrt{8|v|^2+1} \cdot \id_H.
\end{align*}
It holds
\begin{equation}
  \label{esti_B_bdd}
  \begin{aligned}
    \|B(v)\|_{\LB_2(U_0,H)}^2
    &= \sigma^2 \cdot \sum_{j \in \N}
    \|\sqrt{8|v|^2+1}Q^\frac{1}{2}\chi_j\|_H^2\\
    &\leq \sigma^2 \big( \sum_{j \in \N} q_j \|\chi_j\|_{C([0,1])}^2 \big)
    \|\sqrt{8|v|^2+1}\|_{L^2(0,1)}^2\\
    &= 2 \sigma^2 \Tr(Q) (8\|v\|_H^2 + 1)
  \end{aligned}
\end{equation}
for every $v \in V$. This ensures that the operator $B$ is well-defined and
bounded. In addition, notice that the mapping
$g \colon \R \to \R, \, x \mapsto \sqrt{8x^2+1}$ is Lipschitz continuous with
Lipschitz constant $L=\sqrt{8}$. Therefore, we obtain
\begin{equation}
  \label{esti_B_Lip}
  \begin{aligned}
  \|B(v)-B(u)\|_{\LB_2(U_0,H)}^2
  &= \sigma^2 \cdot \sum_{j \in \N}
  \|(\sqrt{8|v|^2+1}-\sqrt{8|u|^2+1})Q^\frac{1}{2}\chi_j\|_H^2\\
  &\leq \sigma^2 \big( \sum_{j \in \N} q_j \|\chi_j\|_{C([0,1])}^2 \big)
  \|g(v)-g(u)\|_{L^2(0,1)}^2\\
  &\leq 16 \sigma^2 \Tr(Q) \|v-u\|_H^2
  \end{aligned}
\end{equation}
for all $v,u \in V$.

Since the operators $A \colon V \to V^\ast$ and $B \colon V \to \LB_2(U_0,H)$
are Lipschitz continuous, both are $\B(V)$-measurable. The Lipschitz continuity
of the operator $A$ also implies that $A$ is hemi-continuous and 
grows linearly with $p=2$. Further, it holds
$\langle A(v),v \rangle_{V^\ast \times V} \geq m_1 \|v\|_V^2$ for all
$v\ \in V$. Together with the estimate \eqref{esti_B_bdd} this yields that the
coercivity condition \eqref{con:coer AB} holds with $\kappa =
16\nu\sigma^2\Tr(Q)$, $c = 2\nu\sigma^2\Tr(Q)$ and $\mu = m_1$ for any
$\nu \in (1,\infty)$. The strong monotonicity of $A$ and the estimate
\eqref{esti_B_Lip} imply that the stronger monotonicity condition
\eqref{con:str mono AB} is satisfied with $K = 2m_1 > 0$. Hence
Assumption~\ref{assump AB2} is fulfilled.

The smooth initial value $X_0 = \sin(\pi \cdot)$ satisfies $X_0 \in V$. However,
in case of this nonlinear problem, sufficient regularity properties of the exact
solution could neither be proven nor found in the literature.

The initial values for the schemes \eqref{BEM scheme} and \eqref{BDF2 scheme} are
computed in the same way as in Subsection \ref{subsec:SHE}. Notice that the operator
$B \colon H \to \LB_2(U_0,H)$ is Lipschitz continuous due to estimate
\eqref{esti_B_Lip} and hence the consistency of the initial values with respect to
the $H$-norm is sufficient for Assumption~\ref{assump IC2} to be satisfied.

As before, the projection of the noise term on the Galerkin space $V_h$ is realized
by applying the interpolation operator $I_h$ such that for each $i = 1,\ldots,N_h$
\begin{align*}
  \big( B(X^{n-1}_{k,h}) \Delta_k W^{n,J},\phi_i \big)_H
  &\approx \big( I_h(B(X^{n-1}_{k,h}) \Delta_k W^{n,J}), \phi_i \big)_H\\
  &= \sigma \sum_{j=1}^{N_h} (\phi_i,\phi_j)_H \cdot
    \big[ B(X^{n-1}_{k,h}(x_j)) \Delta_k W^{n,J}(x_j) \big].
\end{align*}
Since the mapping $g$ is continuously differentiable and $X^{n-1}_{k,h}$ is
$V$-valued, the composition $g \circ X^{n-1}_{k,h}$ is also $V$-valued,
compare, e.g., with \cite[Corollary 8.11]{brezis2011}. Moreover, the discrete
Wiener increment $\Delta_k W^{n,J}$ is smooth and hence the term
$B(X^{n-1}_{k,h}) \Delta_k W^{n,J}$ is $V$-valued as well as continuously embedded
into $C([0,1])$. This ensures that this approximation is well-defined and, by
estimate \eqref{ineq:I_h err L2}, the corresponding interpolation error is 
of order $\mathcal{O}(h)$.

By identifying $X^n_{k,h}$ with the random $\R^{N_h}$-valued 
vector $\boldsymbol{X}^n$ through the 
formula \eqref{FEM representation} for each $n=0,\ldots,N_k$, we define
the stiffness matrix by
\begin{equation*}
  A_h(\boldsymbol{X}^n)
  = \Big[ \big( \psi\big(|(X^n_{k,h})'|\big) \phi_i',\phi_j' \big)_{L^2(0,1)}
    \Big]_{i,j=1}^{N_h}
  \quad \text{with} \quad
  (X^n_{k,h})' = \sum_{l=1}^{N_h} \boldsymbol{X}^n_l \phi_l',
\end{equation*}
and introduce the notation
\begin{equation*}
  B(\boldsymbol{X}^{n-1}) \Delta_k \boldsymbol{W}^{n,J}
  = \Big[ (8|\boldsymbol{X}^{n-1}_i|^2+1)^{\frac{1}{2}} 
    \big(W^J(t_n,x_i)-W^J(t_{n-1},x_i) \big) \Big]_{i=1}^{N_h}.
\end{equation*}
Since $X^n_{k,h} \in V_h$ is piecewise linear, the corresponding derivative
$(X^n_{k,h})'$ is piecewise constant and can be directly implemented without
applying any quadrature. The reduced discrete systems of \eqref{prob:SNP} for the
BEM scheme and the BDF2-Maruyama scheme are given, respectively, by
\begin{equation*}
  M_h (\boldsymbol{X}^n -\boldsymbol{X}^{n-1}) + k A_h(\boldsymbol{X}^n)
    \boldsymbol{X}^n
  = \sigma M_h \big( B(\boldsymbol{X}^{n-1}) \Delta_k \boldsymbol{W}^{n,J} \big)
\end{equation*}
and
\begin{align*}
  &M_h (3\boldsymbol{X}^n - 4\boldsymbol{X}^{n-1} + \boldsymbol{X}^{n-2})
    + 2k A_h(\boldsymbol{X}^n) \boldsymbol{X}^n\\
  &\qquad = \sigma M_h
    \big( 3 B(\boldsymbol{X}^{n-1}) \Delta_k \boldsymbol{W}^{n,J}
    - B(\boldsymbol{X}^{n-2}) \Delta_k \boldsymbol{W}^{n-1,J} \big).
\end{align*}

Since the discrete systems of both schemes are nonlinear, we solve for
$\boldsymbol{X}^n$ by applying Newton's method with $N$ iterations in each
temporal step. In more detail, we set
$\hat{\boldsymbol{X}}^0 \coloneqq \boldsymbol{X}^{n-1}$ and
compute iteratively $\hat{\boldsymbol{X}}^l$ for each $l \in \{1,\ldots,N\}$ by
\begin{equation*}
  J_{k,h}(\hat{\boldsymbol{X}}^{l-1})
    (\hat{\boldsymbol{X}}^l - \hat{\boldsymbol{X}}^{l-1})
  =  -F^n_{k,h}(\hat{\boldsymbol{X}}^{l-1}),
\end{equation*}
where $F^n_{k,h}$ and its Jacobian $J_{k,h}$ are given for the BEM scheme and
the BDF2-Maruyama scheme, respectively, by
\begin{align*}
  F^n_{k,h}(\boldsymbol{X})
  &= \big( M_h + k A_h(\boldsymbol{X}) \big) \boldsymbol{X}
    - M_h (\boldsymbol{X}^{n-1}
    + \sigma B(\boldsymbol{X}^{n-1}) \Delta_k \boldsymbol{W}^{n,J})\\
  J_{k,h}(\boldsymbol{X}) &= M_h + k A^\ast_h(\boldsymbol{X})
\end{align*}
and
\begin{align*}
  F^n_{k,h}(\boldsymbol{X})
  &= \big( 3M_h + 2k A_h(\boldsymbol{X}) \big) \boldsymbol{X}
    - M_h \big( 4\boldsymbol{X}^{n-1} - \boldsymbol{X}^{n-2} \big)\\
  &\qquad - \sigma M_h
    \big( 3 B(\boldsymbol{X}^{n-1}) \Delta_k \boldsymbol{W}^n
    - B(\boldsymbol{X}^{n-2}) \Delta_k \boldsymbol{W}^{n-1,J} \big)\\
  J_{k,h}(\boldsymbol{X}) &= 3M_h + 2k A^\ast_h(\boldsymbol{X}).
\end{align*}
Hereby, $A^\ast_h$ denotes the Jacobian of the mapping $\boldsymbol{X} \mapsto
A_h(\boldsymbol{X}) \boldsymbol{X}$. A short computation yields
\begin{equation*}
  A_h^\ast(\boldsymbol{X})
  = \bigg[ \Big( \big[ \psi(|X'|) + \psi'(|X'|) |X'| \big] \phi_i',\phi_j'
    \Big)_{L^2(0,1)} \bigg]_{i,j=1}^{N_h}
  \quad \text{with} \quad X' = \sum_{l=1}^{N_h} \boldsymbol{X}_l \phi_l'.
\end{equation*}
The number of iterations $N$ is at least $N_{min} = 3$ and is increased up to
$N_{max} = 10$ as long as the current residual exceeds the tolerance limit
$tol = 10^{-12}$.

As before, we consider $T=1$ throughout all numerical experiments. First, we
simulate \eqref{prob:SNP} without noise by setting $\sigma=0$. The corresponding
results in Table~\ref{table:SNP no noise} show that the BDF2-Maruyama scheme
yields a more accurate approximation and converges faster for smaller temporal
step sizes $k$ in comparison to the BEM scheme.

\begin{table}[ht]
  \caption{Nonlinear deterministic PDE \eqref{prob:SNP} with $\sigma=0$.}
  \label{table:SNP no noise}
  \begin{tabular}{p{1cm}p{1.5cm}p{1.2cm}p{1.5cm}p{1.2cm}}
		  & BEM      &		 & BDF2     & \\ 
	\noalign{\smallskip}\hline\noalign{\smallskip}
	$N_k$ & error    & EOC   & error    & EOC \\
	\noalign{\smallskip}\hline\noalign{\smallskip}
	32    & 0.066045 &       & 0.040309 & \\ 
	64    & 0.040482 & 0.71  & 0.025797 & 0.64 \\ 
	128   & 0.022121 & 0.87  & 0.012795 & 1.01 \\ 
	256   & 0.011636 & 0.93  & 0.005395 & 1.25 \\ 
	512   & 0.005986 & 0.96  & 0.002478 & 1.12 \\ 
	1024  & 0.003038 & 0.98  & 0.000994 & 1.32 \\ 
	\noalign{\smallskip}\hline\noalign{\smallskip}
  \end{tabular}
\end{table}

In addition, we simulate \eqref{prob:SNP} with different noise intensities by
choosing $\sigma \in \{0.25,0.75\}$. In Table~\ref{table:SNP sigma=0.25} we
observe that for smaller noise intensity with $\sigma = 0.25$ the approximation
results behave very similar to the deterministic case. In particular, the
BDF2-Maruyama scheme provides more favourable results. On the contrary, we
notice in Table~\ref{table:SNP sigma=0.75} that the advantage of the
BDF2-Maruyama scheme over the BEM scheme for larger noise intensity with
$\sigma = 0.75$ is barely noticeable and diminishes as the temporal step size
decreases.

In conclusion, the numerical experiments indicate that our theoretical results
are indeed applicable to this nonlinear stochastic partial differential equation.
In case of small noise intensity, the BDF2-Maruyama scheme performs significantly
better than the BEM scheme for similar temporal refinement levels. The margin is
less significant for large noise intensity though.

\begin{table}[ht]
  \caption{Nonlinear stochastic PDE \eqref{prob:SNP} with $\sigma=0.25$.}
  \label{table:SNP sigma=0.25}
  \begin{tabular}{p{1cm}p{1.5cm}p{1.55cm}p{1.2cm}p{1.5cm}p{1.55cm}p{1.2cm}}
          & BEM & & 
	      & BDF2 & & \\ 
	\noalign{\smallskip}\hline\noalign{\smallskip}
	$N_k$ & error & CI $\pm$ & EOC
	      & error & CI $\pm$ & EOC \\ 
	\noalign{\smallskip}\hline\noalign{\smallskip}
    32    & 0.067400 & 0.000226 &       & 0.043252 & 0.000250 & \\ 
	64    & 0.041681 & 0.000138 & 0.69  & 0.027827 & 0.000150 & 0.64 \\ 
	128   & 0.022931 & 0.000078 & 0.86  & 0.014031 & 0.000073 & 0.99 \\ 
	256   & 0.012223 & 0.000046 & 0.91  & 0.006340 & 0.000027 & 1.15 \\ 
	512   & 0.006412 & 0.000028 & 0.93  & 0.003093 & 0.000013 & 1.04 \\ 
	1024  & 0.003343 & 0.000018 & 0.94  & 0.001417 & 0.000009 & 1.13 \\   
	\noalign{\smallskip}\hline\noalign{\smallskip}
  \end{tabular}
\end{table}

\begin{table}[ht]
  \caption{Nonlinear stochastic PDE \eqref{prob:SNP} with $\sigma=0.75$.}
  \label{table:SNP sigma=0.75}
  \begin{tabular}{p{1cm}p{1.5cm}p{1.55cm}p{1.2cm}p{1.5cm}p{1.55cm}p{1.2cm}}
		  & BEM & & 
		  & BDF2 & & \\ 
    \noalign{\smallskip}\hline\noalign{\smallskip}
    $N_k$ & error & CI $\pm$ & EOC
	      & error & CI $\pm$ & EOC \\ 
    \noalign{\smallskip}\hline\noalign{\smallskip}
    32    & 0.092583 & 0.000849 &       & 0.081076 & 0.001099 & \\ 
	64    & 0.061146 & 0.000686 & 0.60  & 0.055266 & 0.000790 & 0.55 \\ 
	128   & 0.038248 & 0.000505 & 0.68  & 0.035124 & 0.000648 & 0.65 \\ 
	256   & 0.024065 & 0.000338 & 0.67  & 0.022407 & 0.000389 & 0.65 \\ 
	512   & 0.015567 & 0.000239 & 0.63  & 0.014762 & 0.000259 & 0.60 \\ 
	1024  & 0.010278 & 0.000159 & 0.60  & 0.009875 & 0.000163 & 0.58 \\    
	\noalign{\smallskip}\hline\noalign{\smallskip}
  \end{tabular}
\end{table}

\section*{Acknowledgment}

The authors like to thank Etienne Emmrich for very helpful comments on
the BDF2 method for nonlinear evolution equations.

The first part of this research was carried out in the framework
of \textsc{Matheon} supported by Einstein Foundation Berlin. It was also
financially supported by TU Berlin (ASF Nr.~3306). 
RK also gratefully acknowledges financial support by the German Research
Foundation (DFG) through the research unit FOR 2402 -- Rough paths,
stochastic partial differential equations and related topics -- at TU Berlin.

\end{document}